\documentclass{article}

\usepackage[preprint]{neurips_2026}

\usepackage[utf8]{inputenc} 
\usepackage[T1]{fontenc}    
\usepackage{hyperref}       
\usepackage{url}            
\usepackage{booktabs}       
\usepackage{amsfonts}       
\usepackage{nicefrac}       
\usepackage{microtype}      
\usepackage{xcolor}         
\usepackage{amsmath}
\usepackage{booktabs}
\usepackage{amssymb}
\usepackage{amsthm}
\usepackage{float}
\usepackage{tikz} 
\usepackage{comment}
\usepackage{pgfplots}
\usepackage{algorithm,algorithmic}
\newtheorem{theorem}{Theorem}
\newtheorem{proposition}{Proposition}
\newtheorem{definition}{Definition}
\newtheorem{corollary}{Corollary}
\newtheorem{lemma}{Lemma}
\newtheorem{remark}{Remark}
\newtheorem{assumption}{Assumption}[section]

\newcommand{\R}{\mathbb{R}}
\newcommand{\E}{\mathbb{E}}

\newcommand{\prox}{\ensuremath{\operatorname{prox}}}
\newcommand{\zprox}{\ensuremath{\operatorname{zprox}}}
\newcommand{\szprox}{\ensuremath{\widehat{\operatorname{zprox}}}}

\newcommand{\dist}{\ensuremath{\operatorname{dist}}}

\newcommand{\Fix}
{\mathrm{Fix}}

\newcommand{\osc}{\ensuremath{\operatorname{osc}}}

\DeclareMathOperator*{\argmin}{arg\,min}

\title{Convergence of zeroth-order proximal point algorithms in the high-temperature regime}

\author{%
  Emanuele Naldi \\
  MaLGa, DIMA\\
  Universit\`a degli Studi di Genova\\
  Genoa, Italy \\
  \texttt{emanuele.naldi@edu.unige.it} \\
  \And
  Hippolyte Labarri\`ere \\
  MaLGa, DIBRIS\\
  Universit\`a degli Studi di Genova\\
  Genoa, Italy \\
  \texttt{hippolyte.labarriere@edu.unige.it} \\
  \And
  Cesare Molinari \\
  MaLGa, DIMA\\
  Universit\`a degli Studi di Genova\\
  Genoa, Italy \\
  \texttt{cesare.molinari@edu.unige.it} \\
  \And
  Silvia Villa \\
  MaLGa, DIMA\\
  Universit\`a degli Studi di Genova\\
  Genoa, Italy \\
  \texttt{silvia.villa@unige.it} 
}

\begin{document}

\maketitle

\begin{abstract}
Efficient methods for non-convex black-box optimization largely rely on sampling. In this context, the Zeroth-Order Proximal Operator (ZOPO) and the corresponding Zeroth-Order Proximal Point Algorithm (ZOPPA) have attracted significant interest, as they combine the advantage of requiring only objective evaluations with the powerful theoretical framework of proximal point algorithms. ZOPO depends on a temperature parameter which, when going to zero, reduces ZOPO to the exact proximal operator. By exploiting this property, the vanishing-temperature regime has been leveraged in several works to obtain theoretical guarantees via inexact proximal methods. However, this regime is computationally unsustainable when sampling is used to estimate ZOPO, since the corresponding Monte Carlo estimators suffer from severe variance issues. We therefore propose a comprehensive analysis of ZOPO for any fixed positive temperature, and prove convergence of ZOPPA under minimal assumptions on the objective function. We do so by demonstrating that ZOPPA can be interpreted as an exact proximal point method applied to an auxiliary smoothed objective, rather than an inexact method on the original function. Importantly, we further derive explicit guarantees connecting this smoothed problem back to the original objective and establish convergence results for the sampled method (S-ZOPPA) at a fixed temperature.
\end{abstract}

\section{Introduction}\label{sec:sbnsm}

We consider the problem of minimizing a function through pointwise evaluations only. More precisely, given an objective $f:\mathbb{R}^d\to\mathbb{R}$, we seek to solve
\begin{equation}
\label{eq:P}\tag{P}
    \min_{x\in\mathbb{R}^d} f(x),
\end{equation}
in a setting where derivatives of $f$ are unavailable, or too costly to compute. Such problems are commonly referred to as zeroth-order, derivative-free, or black-box optimization problems \cite{Spall,Conn2009,Larson2019}. They arise whenever the objective is specified implicitly through a simulation or a complex computational pipeline, so that optimization must proceed from function values alone \cite{Salimans2017EvolutionSA,Mania18,pmlr-v80-choromanski18a,flaxman2005online}.
In machine learning, black-box optimization arises for instance in adversarial attacks \cite{chen2017zoo} and  to overcome the extreme memory requirements of standard backpropagation during LLMs fine tuning \cite{zhang2024revisiting}.
A common strategy in zeroth-order optimization is to use function evaluations to recover, either explicitly or implicitly, local first-order information \cite{Spall,Conn2009}. Finite-difference and random-direction methods \cite{Conn2009,nest} follow this principle by constructing gradient surrogates and then embedding them in variants of gradient descent. This viewpoint has led to a rich class of algorithms and convergence guarantees. However, gradient descent is only one of the basic mechanisms underlying first-order optimization \cite{BeckBook}. Another central object for first order optimization is the proximal operator, which provides stable implicit updates and plays a fundamental role in optimization, variational analysis, and splitting methods \cite{combettes2005signal}. The question whether it is possible to design zeroth order analogues of proximal algorithms has been recently answered in \cite{osher2023hamilton} (see also \cite{crandall1984}), where the Zeroth-Order Proximal Operator (ZOPO) has been introduced. Given parameters \(\lambda>0\) and \(\delta>0\), ZOPO is
defined by
\begin{equation}\label{eq:def_zprox_intro}
   \zprox^\delta_{\lambda, f}(x)
    =
    \frac{
    \mathbb{E}_{y\sim\mathcal{N}(x,\lambda\delta I)}
    \left[y\exp(-f(y)/\delta)\right]
    }{
    \mathbb{E}_{y\sim\mathcal{N}(x,\lambda\delta I)}
    \left[\exp(-f(y)/\delta)\right]
    }.
\end{equation}
This operator replaces the variational minimization defining the classical proximal map by a
Gibbs-weighted average of zeroth-order queries around the current point. Points with lower objective
value receive exponentially larger weight, yielding an implicit descent mechanism that uses only
function evaluations. The ideal zeroth-order proximal point method, which we call ZOPPA \cite{osher2023hamilton} (also introduced as Consensus Hopping (CH) in \cite{Fornasier2024}), iterates
\begin{equation}\label{eq:ZOPPA_intro}
    x^{k+1}
    =
    \zprox^\delta_{\lambda,f}(x^k).
\end{equation}
Although this update is conceptually simple, the expectations in \eqref{eq:def_zprox_intro} are
typically intractable. We therefore consider its Monte Carlo approximation, the Stochastic
Zeroth-Order Proximal Point Algorithm (S-ZOPPA):
\begin{equation}\label{eq:S-ZOPPA_intro}
    x^{k+1}
    =
    \frac{
    \sum_{i=1}^N y_i \exp(-f(y_i)/\delta)
    }{
    \sum_{i=1}^N \exp(-f(y_i)/\delta)
    },
    \qquad
    y_i\sim\mathcal{N}(x^k,\lambda\delta I).
\end{equation}
Thus, ZOPPA is the population-level proximal iteration, while S-ZOPPA is its practical sampled implementation. This distinction allows us to separate the deterministic properties of the zeroth-order proximal map from the statistical error introduced by Monte Carlo sampling.

\paragraph{The high-temperature regime.} The main contribution of this work is to focus on the behavior of ZOPPA and S-ZOPPA when the temperature parameter \(\delta\) is bounded away from \(0\). This perspective distinguishes this paper from the increasing number of works \cite{osher2023hamilton, zhang2024inexact, di2026operator, Fornasier2026} that build on the vanishing-temperature regime and the inexact proximal operator viewpoint and is particularly motivated by sampling issues. Indeed, S-ZOPPA is built around a Monte Carlo estimation of ZOPO \eqref{eq:def_zprox_intro} which has the property of tending to the exact value of the proximal operator when \(\delta\to0\) (see the orange curve on Figure~\ref{fig:intro_sampling} (c)). However, estimating it turns back to approximating  \(\int \exp(-f(y)/\delta)\exp(-\|x-y\|^2/2\lambda\delta) dy\) which is increasingly hard as \(\delta\) goes to \(0\) (see (a) and (b) in Figure~\ref{fig:intro_sampling}). This ultimately leads to extremely inaccurate approximations when \(\delta\) is small, for Monte Carlo estimation with a fixed number of samples, as shown by the blue curve in Figure~\ref{fig:intro_sampling} (c). This illustrates that, already in this simple example, the most reliable sampling regime is not the vanishing-temperature limit, but an intermediate range of positive temperatures where ZOPO remains close to the proximal point while its Monte Carlo approximation is still stable.

\begin{figure}[htbp]
    \centering
    \tikzset{
        mathfuncs/.style={
            declare function={
                f(\y) = abs(\y);
                target(\y,\d) = exp(-f(\y)/\d);
                gauss(\y,\x,\d,\l) = exp(-(\y-\x)^2 / (2*\l*\d)); 
            }
        }
    }
    
    \def\xval{2}
    \def\lambdaval{1}
    \def\deltalarge{1}
    \def\deltasmall{0.01}
    \def\trueprox{1}

    \begin{minipage}[t]{0.31\linewidth} 
        \centering
        (a) Large \(\delta\)\par\vspace{0.15cm}
        \begin{tikzpicture}[mathfuncs]
        \begin{axis}[
            width=\linewidth, 
            height=4.5cm, 
            xlabel={$y$},
            ymin=-0.05, ymax=1.15,
            xmin=-2.5, xmax=4.5,
            domain=-2.5:4.5,
            samples=2000,
            axis lines=left
        ]
            \addplot [very thick, red] {target(x,\deltalarge)};
            \addplot [very thick, blue, dashed] {gauss(x,\xval,\deltalarge,\lambdaval)};
        \end{axis}
        \end{tikzpicture}
    \end{minipage}%
    \hfill
    \begin{minipage}[t]{0.31\linewidth} 
        \centering
        (b) Small \(\delta\)\par\vspace{0.15cm}
        \begin{tikzpicture}[mathfuncs]
        \begin{axis}[
            width=\linewidth, 
            height=4.5cm, 
            xlabel={$y$},
            ymin=-0.05, ymax=1.15,
            xmin=-2.5, xmax=4.5,
            domain=-2.5:4.5,
            samples=5000, 
            axis lines=left
        ]
            \addplot [very thick, red] {target(x,\deltasmall)};
            \addplot [very thick, blue, dashed] {gauss(x,\xval,\deltasmall,\lambdaval)};
        \end{axis}
        \end{tikzpicture}
    \end{minipage}%
    \hfill
    \begin{minipage}[t]{0.34\linewidth} 
        \centering
        (c) Impact on sampling\par\vspace{0.15cm}
        \includegraphics[width=\linewidth]{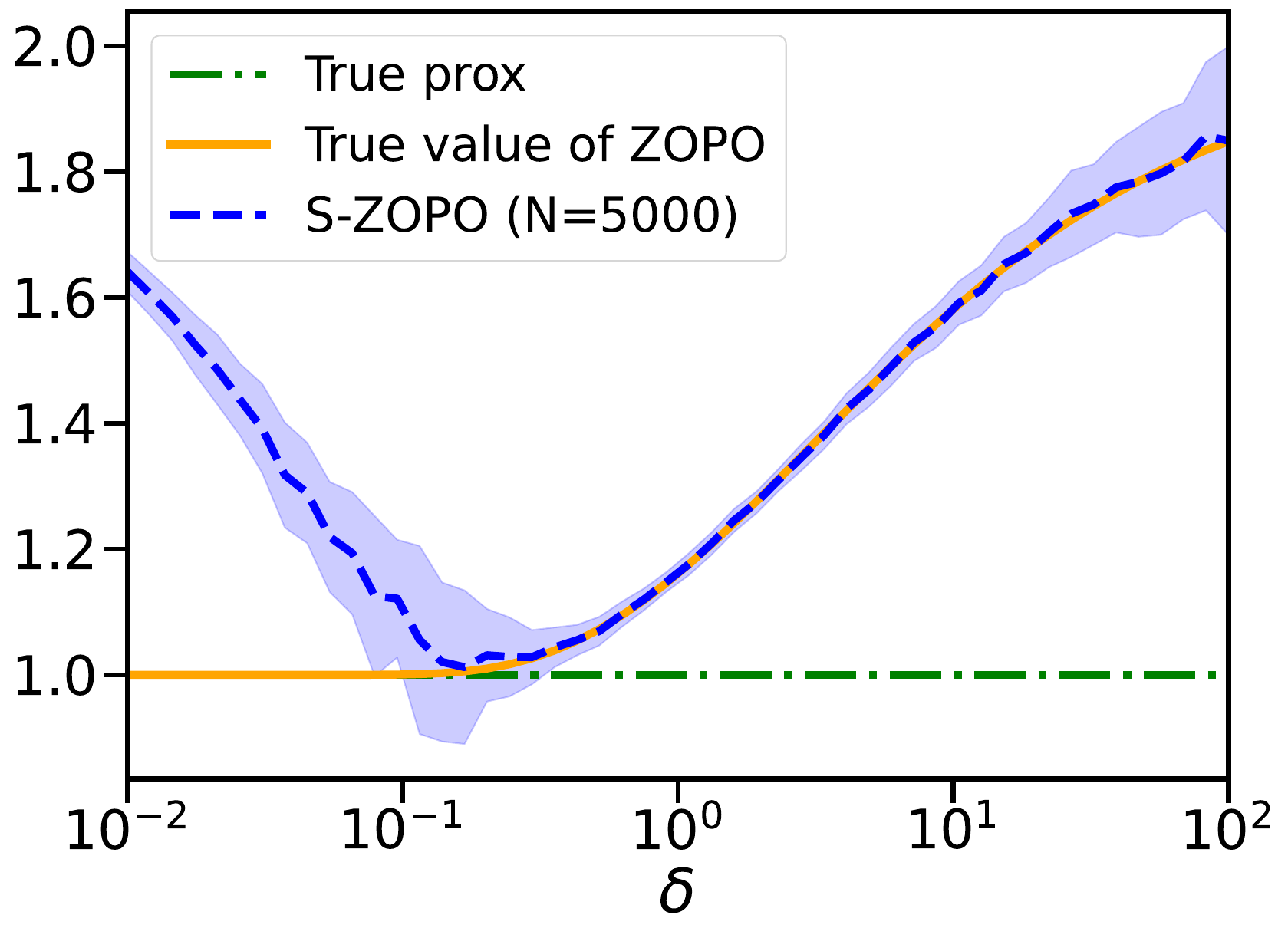}
    \end{minipage}

    \caption{Sampling difficulty as \(\delta\) decreases for \(f:x\mapsto |x|\), with \(x=2\) and \(\lambda=1\). Left and center: comparison of the objective weight \(\exp(-f/\delta)\) and the Gaussian proposal weight for large and small \(\delta\). Right: classical prox, exact ZOPO, and S-ZOPO approximation as functions of \(\delta\).}
    \label{fig:intro_sampling}
\end{figure}
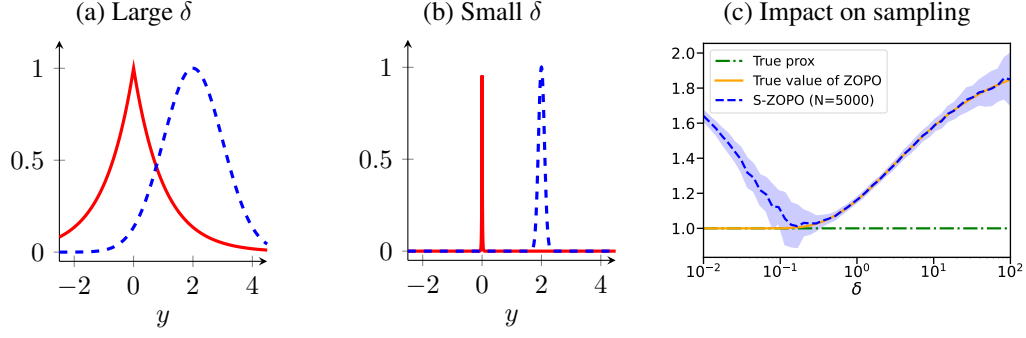

\paragraph{Contributions.}
Our contributions can be summarized as follows.

\begin{itemize}
    \item \textbf{High-temperature analysis of ZOPPA.}
    We provide the first convergence analysis of ZOPPA at fixed positive temperature in the nonconvex setting, under minimal assumptions on \(f\). In contrast to previous works, our analysis does not view ZOPO as an inexact approximation of the proximal operator of \(f\) but as an exact proximal point method applied to an auxiliary objective related to the soft Moreau envelope of \(f\). This allows us to show that the approximation error does not accumulate along the iterations, and therefore that no vanishing-temperature schedule is required. 
    \item \textbf{Connection with the original objective \(f\).}
    Although our analysis is carried out through a smoothed auxiliary objective, our goal is the minimization of the original function \(f\). To this end, we establish links between the behavior of the iterates \((x^k)_k\) on the smoothed landscape and approximate stationarity for \(f\), by controlling quantities such as \(\|\nabla f(x^k)\|\) along the trajectory. In the convex setting, we further control the function values \(f(x^k)-\min f\).
    \item \textbf{Role of the stepsize parameter \(\lambda\).} We investigate in detail the effect of the parameter \(\lambda\). On the theoretical side, we prove a convexification result showing that large values of \(\lambda\) can simplify the optimization landscape. In practice, this motivates a continuation strategy in which \(\lambda\)
    is decreased over time: large values help avoid poor local minima, while smaller values improve the approximation of stationary points of the original objective.
    \item \textbf{Convergence of S-ZOPPA and sampling considerations.}
    For the sampled algorithm S-ZOPPA, we establish the first convergence result in the convex setting for a fixed high temperature. We further analyze the sampling behavior of the method, both theoretically and experimentally. In particular, using effective sample size (ESS), we show that in the high-temperature regime  ZOPO can be estimated reliably, whereas the small-\(\delta\) regime degenerates toward a form of random search.
\end{itemize}
The paper is organized as follows: in Section~\ref{sec:setting} we introduce the setting and the ZOPPA algorithm. Section \ref{sec:main_results} is the core of the paper: there we state the main convergence results for ZOPPA and S-ZOPPA in the high temperature regime. We consider nonconvex and convex settings.  Moreover we describe the convexification effect induced by a large $\lambda$. Finally, in Section~\ref{sec:experiments} we present numerical experiments clarifying the roles of $\lambda$ and $\delta$ in practice and suggest an effective sampling procedure to implement S-ZOPPA. 
\section{Problem setting and algorithm }\label{sec:setting}
We consider the problem of minimizing $f$ in \eqref{eq:P}, under the following standing assumptions.
\begin{assumption}\label{ass:on_f} Fix $\delta>0$. 
    \begin{itemize} 
        \item [A1)] The function $f:\R^d \to \R$ is continuous and has at least a minimizer. 
        \item[A2)] It holds $C_\delta:=\int \exp\left(-\frac{f(y)}{\delta}\right) dy <+\infty$, i.e. $\exp\left(-\frac{f}{\delta}\right)\in L^1(\mathbb{R}^d)$.
    \end{itemize}
\end{assumption}
Assumption~\ref{ass:on_f} is satisfied, for instance, whenever \(f\) is continuous and grows faster than logarithmically at infinity, namely $\lim_{\|x\|\to\infty}\frac{f(x)}{\log(\|x\|)}=+\infty$, which is a really mild assumption. Under Assumption~\ref{ass:on_f}, for given $\delta>0$ and $\lambda>0$, the operator $\zprox_{\lambda,f}^\delta: \R^d \to \R^d$ defined in \eqref{eq:def_zprox_intro}, is well defined and we consider the following algorithm. 

\begin{algorithm}[H] 
\caption{ZOPPA: Zeroth-Order Proximal Point Algorithm} \label{alg:zoppa}
\begin{algorithmic}
  \STATE{{\bf Input:} $x^0 \in \R^d$, $\delta \in \R_{>0}$, $\lambda\in \R_{>0}$}
     \FOR{$k = 0, 1, \cdots$}
         \STATE{$x^{k + 1} = \zprox^{\delta}_{\lambda, f}(x^k)$}
     \ENDFOR
 \end{algorithmic}    
 \end{algorithm}
ZOPPA has been introduced in \cite{osher2023hamilton} and is inspired by the proximal point algorithm \cite{martinet1970regularisation,rockafellar1976monotone}, but it has the advantage that it can be approximated by Monte Carlo sampling in practice, using only function evaluations. ZOPO is strictly connected with viscous Hamilton-Jacobi equations (see also Appendix \ref{app:HamiltonJacobi}). It depends on the two parameters $\lambda$ and $\delta$, playing a different role, whose choice is crucial as discussed above. Interestingly, we can define an analogue of the Moreau envelope (see Appendix~\ref{app:HamiltonJacobi} for the definition), that we name \emph{soft Moreau envelope}:
\[\forall x\in\R^d,\quad f^{\lambda,\delta}(x)=-\delta \log \E_{y\sim\mathcal N(x,\lambda\delta I)}
\left[
\exp\left(-f(y)/\delta\right)
\right],\]
and the ZOPO is then computed by analogy with the formula relating the gradient of the Moreau envelope and the proximity operator, namely
\[\forall x \in\R^d,\quad \zprox^{\delta}_{\lambda, f}(x)=x-\lambda\nabla f^{\lambda,\delta}(x).\]
although this interpretation is already meaningful, without any regularity of $\nabla f^{\lambda,\delta}$ it is not sufficient to show convergence of ZOPPA. For this reason we take a step further. From a technical perspective one of our main contributions is the interpretation of ZOPPA as a Minimum Mean Square Error Estimator (MMSE) and as a proximal point algorithm applied to another function, building on the results in \cite{Gribonval11}. This point of view is crucial in our analysis since it allow us to derive convergence results also in the high-temperature regime, namely for non-vanishing $\delta$. 

\begin{theorem}\label{thm:zoppa_is_prox}
    Suppose Assumption~\ref{ass:on_f} is satisfied. Then, $\zprox^{\delta}_{\lambda, f}$ is $C^{\infty}$, bijective and \begin{equation}
     \zprox^{\delta}_{\lambda, f}= \prox_{\lambda \delta H_{\lambda,\delta}},
\end{equation}
with $H_{\lambda,\delta}:\R^d \to \R $ defined by
\begin{equation}
H_{\lambda,\delta}(x)= 
    -\frac{1}{2\lambda\delta}\big\|(\zprox^{\delta}_{\lambda, f})^{-1}(x)-x\big\|^2+\frac{1}{\delta}f^{\lambda,\delta}\!\big((\zprox^{\delta}_{\lambda, f})^{-1}(x)\big).
\label{eq:H_def}
\end{equation}
In particular, the sequence $(x^k)_{k\in\mathbb{N}}$ generated by
Algorithm \ref{alg:zoppa} is a proximal point iteration.
\end{theorem}
The proof is given in Appendix \ref{app:proof_of_zoppa_is_prox}. The soft Moreau envelope and Theorem~\ref{thm:zoppa_is_prox} are the main tools for our analysis. We use the interpretation as proximal point algorithm to prove convergence to fixed points of ZOPO under the general Assumption~\ref{ass:on_f} on $f$, then, thanks to the interpretation as gradient descent we characterize the limit points as stationary points of the soft Moreau envelope $f^{\lambda, \delta}$.

\section{Main results}\label{sec:main_results}

\subsection{Convergence of ZOPPA}
With the tools discussed above, we are able to prove the following convergence result for ZOPPA, which is based solely on Assumption~\ref{ass:on_f}. The proof is given in Appendix \ref{app:proof_of_zoppa_convergence}.

\begin{theorem}\label{thm:zoppa_convergence}
Let $f:\R^d\to \R$ satisfy Assumption~\ref{ass:on_f}. Then, the sequence $(x^k)_{k\in\mathbb{N}}$ generated by Algorithm \ref{alg:zoppa} converges to some $x_{\lambda,\delta}^\star$ such that $x_{\lambda,\delta}^\star=\zprox_{\lambda, f}^{\delta}(x_{\lambda,\delta}^\star)$, i.e.,
\[x^k \to x_{\lambda,\delta}^\star, \quad \text{with }\  \nabla f^{\lambda,\delta}(x_{\lambda,\delta}^\star)=0.\]
Moreover, \(\|\nabla f^{\lambda,\delta}(x^k)\|\to 0\) and  \(\sum_k \|\nabla f^{\lambda,\delta}(x^k)\|< +\infty.\)
\end{theorem}

The previous theorem shows not only convergence of ZOPPA, but also that the limit points are meaningful: they are stationary points of the auxiliary objective \(f^{\lambda,\delta}\). This is relevant because \(f^{\lambda,\delta}\) can be interpreted as a smoothed version of \(f\), and it recovers \(f\) in the limit \(\lambda\to 0\). At the same time, our goal remains the minimization of the original function \(f\). The next theorem therefore quantifies how stationarity for \(f^{\lambda,\delta}\) transfers to stationarity for \(f\). The proof is given in Appendix \ref{app:proof_of_rates}.

\begin{theorem}[Rates]\label{thm:rates}
    Let $f:\R^d\to \R$ satisfy Assumption~\ref{ass:on_f} and be $L$-smooth. Let  $(x^k)_{k\in\mathbb{N}}$ be the sequence generated by Algorithm \ref{alg:zoppa}. Then, 
    \begin{enumerate}
        \item[(i)] if $\lambda < 1/L$, it holds
        \[\|\nabla f(x^k)\| \leq (1+\lambda L)\| \nabla f^{\lambda,\delta}(x^k)\| + L \sqrt{\frac{d \lambda \delta}{1-\lambda L}}.\]
        \item[(ii)] if $f$ is also $G$-Lipschitz, for any $\lambda >0$ it holds
\[\begin{aligned}\|\nabla f(x^k)\| \leq & (1+\lambda L)\| \nabla f^{\lambda,\delta}(x^k)\| \\
& \hspace{-0.5cm} + L\,
\min\left\{
\sqrt{2d \lambda\delta}\,
\exp\!\Bigl(2\sqrt{2d/\pi}\,G^2\lambda^2\Bigr),\;
\left(\sqrt{2 d}G\lambda+2\sqrt{d\lambda\delta}\right)\,
\exp\!\Bigl(\frac{G^2\lambda}{4\delta}\Bigr)
\right\}.
\end{aligned}\]
\item[(iii)] if $f(y)=\frac{\kappa}{2}\|y\|^2+V(y)$ with $\operatorname{osc}(V)= \sup V - \inf V <+\infty$, for any $\lambda>0$ it holds
\[
\|\nabla f(x^k)\|
\le
(1+\lambda L)\|\nabla f^{\lambda,\delta}(x^k)\|
+
L\sqrt{
d\,\exp{(\operatorname{osc}(V)/\delta)}
\frac{\lambda\delta}{1+\kappa\lambda}
}.
\]
    \end{enumerate}
\end{theorem}

   Theorem~\ref{thm:rates} establishes an upper bound not suffering from error accumulation. The magnitude of  $\|\nabla f(x^k)\|$ is controlled by an asymptotically vanishing term plus a fixed constant, related to the error due to the approximation of the Moreau envelope and the proximity operator. As can be seen directly from the expression, this additional term depends on  $\delta$ and $\lambda$. In i) and ii) the term vanishes if $\lambda\delta\to 0$ and the upper bound in iii) suggests to keep $\delta$ away from zero. The bound in item $i)$ holds for $\lambda$ sufficiently small, while the bound in item $ii)$ and $iii)$ is valid for any $\lambda$, under further hypothesis on $f$. The results obtained in Theorem~\ref{thm:rates} can be contrasted to the ones implied by classical results about convergence of the proximal  point algorithm with errors \cite{Salzo2012,schmidt2011convergence}, which usually give an upper bound where the error at each step accumulates along the iterations. Indeed, classical analysis requires the errors to be summable in order to derive convergence. This is the crucial difference with our results, which show that for this specific approximation of the proximal operator, the algorithm does not  suffer from error accumulation.  In the {\em strongly convex case}, a convergence result to an approximate solution can be found in \cite[Section 4]{lauga2026}.  

\paragraph{Sketch of proof.} The proof is made of several steps as detailed below. 

{\bf Step 1 - Descent properties.} By the proximal representation established in Theorem
\ref{thm:zoppa_is_prox}, the iteration satisfies the standard proximal descent inequality for $H_{\lambda,\delta}$, i.e.,
\[H_{\lambda,\delta}(x^{k+1})\le H_{\lambda,\delta}(x^{k}) - \frac{1}{2\lambda \delta}\|x^{k+1}-x^{k}\|^2 \text{ for all } k\ge 0. \] 
From the expression of $H_{\lambda,\delta}$ in \eqref{eq:H_def} we derive the descent inequality for the soft Moreau envelope:
\begin{equation}
    f^{\lambda,\delta}(x^{k})\le f^{\lambda,\delta}(x^{k-1}) - \frac{1}{2\lambda}\|x^{k}-x^{k-1}\|^2 \text{ for all } k\ge 1. \tag{\text{Descent}}
\end{equation}
Consequently, $\sum_{k\ge 1}\|x^k-x^{k-1}\|^2<+\infty,$
and hence $\|\nabla f^{\lambda,\delta}(x^k)\|=\|x^{k+1}-x^k\|/\lambda\to 0$ as $k\to \infty$.

{\bf Step 2 - Properties of the soft Moreau envelope.}
It is possible to see that $f^{\lambda,\delta}(x)=
-\delta\log u^\delta(\lambda,x)$ where $u^\delta(\lambda,x)$ is the solution at time $\lambda$ of an heat equation with initial datum $e^{-f/\delta}$ \cite{osher2023hamilton}. This point of view provides two key properties. \textbf{Analyticity and KL inequality:}  $f^{\lambda,\delta}$  is analytic \cite[Section 2.3]{evans} and satisfies the {\L}ojasiewicz gradient inequality around each critical point \cite{Lojasiewicz1963}. \textbf{Coercivity:} since $e^{-f/\delta}$ is positive and belongs to $L^1\cap L^\infty$ the function $u^\delta(\lambda,x)$ vanishes at infinity \cite[Proposition 8.8]{folland1999real}. Therefore, \(f^{\lambda,\delta}(x)\to+\infty\) as \(\|x\|\to+\infty\), so
\(f^{\lambda,\delta}\) is coercive.

{\bf Step 3 - Finite length property}  Since \(f^{\lambda,\delta}(x^k)\) is
decreasing, the sequence \(\{x^k\}_k\) is bounded. Combining the sufficient descent estimate with $\|\nabla f^{\lambda,\delta}(x^k)\| =\|x^{k+1}-x^k\|/\lambda$,
the continuity of $f^{\lambda, \delta}$ and the KL inequality, the standard KL convergence theorem \cite[Theorem 2.9]{attouch2013convergence} yields the finite-length property $\sum_k \|x^{k+1}-x^k\| < +\infty$.  Thus \(\{x^k\}_k\) is Cauchy and converges to some \(x_{\lambda,\delta}^\star\). 

{\bf Step 4 - Characterization of the limit point.}
From the continuity of $\zprox_{\lambda, f}^\delta$ and from
\(x^{k+1}-x^k\to 0\) we obtain
$\zprox_{\lambda, f}^\delta(x_{\lambda,\delta}^\star)=x_{\lambda,\delta}^\star$ and equivalently $
\nabla f^{\lambda,\delta}(x_{\lambda,\delta}^\star)=0$.

\subsection{Discussion on the parameter $\lambda$ and convexification effect}

The bounds in Theorem~\ref{thm:rates} highlight the role of the parameter $\lambda$. Small values of $\lambda$ yield sharper guarantees on $\|\nabla f(x^k)\|$, as the approximation error vanishes with $\lambda$. We show here that, conversely, larger values of $\lambda$ mitigate nonconvexity and simplifies the geometry of the objective.

\begin{theorem}[Convexification for large $\lambda$]\label{thm:convexification}
Let $f$ satisfy Assumption~\ref{ass:on_f}. The following holds.
\begin{enumerate}
    \item[(i)] If $f\in C^1$ satisfies the dissipativity condition $\langle \nabla f(y), y \rangle \geq m\|y\|^2 - b$ for $m > 0$ and $b \geq 0$, then, for any radius $R > 0$, $f^{\lambda,\delta}$ is convex on $B(0, R)$ provided
    \begin{equation}
        \lambda \geq \frac{\delta d + b + \sqrt{(\delta d + b)^2 + 2\delta m R^2}}{2\delta m}.
    \end{equation}
        \item[(ii)] If $f(y) = \frac{\kappa}{2}\|y\|^2 + V(y)$ for some $\kappa > 0$, where $V$ is bounded with global oscillation $\operatorname{osc}(V) := \sup V - \inf V < \infty$, then, $f^{\lambda,\delta}$ is globally convex provided
\begin{equation}
        \lambda \geq \frac{e^{\osc(V)/\delta}-1}{\kappa}.
    \end{equation}
\end{enumerate}
\end{theorem}
    The proof is given in Appendix \ref{app:proof_of_convexification}.
As we can see from the statement, the smaller $\delta$, the bigger should be $\lambda$ to ensure convexity of $f^{\lambda,\delta}$.
\begin{remark} 
    The dissipativity condition in Theorem~\ref{thm:convexification}(i) is a standard assumption in the literature about non-convex optimization and sampling, particularly within the analysis of Langevin dynamics, see e.g., \cite{Raginsky2017, XCZG2018}. While the resulting convexification is restricted to a ball of radius $R$, this is sufficient for our purposes, as the sequence of iterates $(x^k)_{k\in\mathbb{N}}$ generated by Algorithm \ref{alg:zoppa} is guaranteed to be bounded.
\end{remark}

The role of $\lambda$ as progressively removing negative curvature and simplifying the landscape of $f^{\lambda,\delta}$, is further shown in Figure~\ref{fig:convexification_app} of Appendix \ref{app:proof_of_convexification} via some numerical illustrations.

\subsection{The convex case}

If $f$ is convex we can directly prove convergence of the method and give much stronger rates. Actually, the following theorem works under the weaker assumption of $f^{\lambda,\delta}$ being convex. As mentioned previously, $f^{\lambda,\delta}$ is coercive and smooth, therefore $\argmin f^{\lambda,\delta} \neq \emptyset$. 

\begin{theorem}\label{thm:zoppa_cvx}
Let $f:\mathbb{R}^d\to\mathbb{R}$ satisfying Assumption~\ref{ass:on_f} and such that $f^{\lambda,\delta}$ is convex and denote $X_{\lambda,\delta}^{\star}:=\argmin f^{\lambda,\delta}$. Then there exists a convex function $H_{\lambda,\delta}:\mathbb{R}^d\to\mathbb{R}$ such that
\[
\zprox_{\lambda, f}^{\delta}=\prox_{\lambda\delta H_{\lambda,\delta}}.
\]
In particular, $\zprox_{\lambda, f}^{\delta}$ is firmly nonexpansive \cite{combettes2005signal}. Moreover, the sequence $(x^k)_{k\in\mathbb{N}}$ converges to some $x^\star_{\lambda,\delta}\in X_{\lambda,\delta}^{\star}$, and for every $k\ge1$,
\begin{equation}\label{eq:GDrates}
    f^{\lambda,\delta}(x^k)-\min f^{\lambda,\delta}
\le
\frac{\dist(x^0,X_{\lambda,\delta}^{\star})^2}{2\lambda k}, \quad \|\nabla f^{\lambda,\delta}(x^k)\|
\le
\frac{2\,\dist(x^0,X_{\lambda,\delta}^{\star})}{\lambda(k+1)}.
\end{equation}
\end{theorem}
        The proof is given in Appendix \ref{app:proof_of_zoppa_cvx}.

\begin{corollary}
The conclusion of the previous theorem still holds if the assumption that $f^{\lambda,\delta}$ is convex is replaced by the requirement that $f$ is convex or satisfies the conditions of Theorem~\ref{thm:convexification}(ii).
\end{corollary}

\begin{theorem}[Rates]\label{thm:zoppa_cvx_rates}
Let $f:\mathbb{R}^d\to\mathbb{R}$ be a convex function satisfying Assumption~\ref{ass:on_f} and denote $X^{\star}:=\argmin f$. Then \begin{enumerate}
    \item[(i)] \begin{equation}\label{eq:rate_nonsmooth}
        f(x^{k})-\min_y f(y)\leq \|\nabla f^{\lambda,\delta}(x^{k-1})\|\operatorname{dist}(x^{k},X^\star)+ d \delta.
    \end{equation}
    \item[(ii)] if $f$ is $L$-smooth then
    \(\|\nabla f(x^k)\| \leq (1+\lambda L) \|\nabla f^{\lambda, \delta}(x^k)\|+ L \sqrt{d \lambda \delta}\)
    and additionally to \eqref{eq:rate_nonsmooth} it also holds
    \[f(x^k) - \min_y f(y) \leq \left((1+\lambda L) \|\nabla f^{\lambda, \delta}(x^k)\|+ L \sqrt{d \lambda \delta}\right)\dist(x^k,X^\star).\]
\end{enumerate}
In all the previous equations $\|\nabla f^{\lambda,\delta}(x^k)\|$ can be bounded as in \eqref{eq:GDrates}.
\end{theorem}
    The proof is given in Appendix \ref{app:proof_of_zoppa_cvx_rates}.

\subsection{Convergence of S-ZOPPA}

In practice, the exact value of \(\zprox^{\delta}_{\lambda, f}\) at any point is not available. It is then natural to approximate it by Monte Carlo sampling \cite{osher2023hamilton,tibshirani2025laplace,luo2025asymptotic,di2026operator,zhang2024inexact,morales2026convergenceratesstochasticproximal}, defining the following estimator for \(N\) samples
\begin{equation}\label{eq:S-ZOPO}
    \forall x\in\R^d,\quad\szprox^{\delta,N}_{\lambda, f}(x)=\frac{\sum_{i=1}^Ny_i\exp(-f(y_i)/\delta)}{\sum_{i=1}^N\exp(-f(y_i)/\delta)},\quad y_i\sim\mathcal{N}\left(x,\lambda\delta I\right).
\end{equation}
This estimator of \(\zprox^{\delta}_{\lambda, f}\) can be related to importance sampling and is known to be biased. We give a more detailed description of its behavior in Section \ref{sec:sampling_challenges}. Based on this estimator, we can build the stochastic analog of ZOPPA.

\begin{algorithm}[H] 
\caption{S-ZOPPA: Stochastic Zeroth-Order Proximal Point Algorithm} \label{alg:szd}
\begin{algorithmic}
  \STATE{{\bf Input:} $x^0 \in \R^d$, $\delta \in \R_{>0}$, $\lambda\in \R_{>0}$, $\left(N_k\right)_{k\in\mathbb{N}}\in \mathbb{N}^\mathbb{N}$}
     \FOR{$k = 0, 1, \cdots$}
         \STATE{$x^{k + 1} = \szprox^{\delta,N_k}_{\lambda, f}(x^k)$}
     \ENDFOR
 \end{algorithmic}    
 \end{algorithm}

The following convergence results hold for S-ZOPPA, supposing that $f^{\lambda,\delta}$ is convex (a weaker condition than $f$ convex).

\begin{theorem}\label{thm:S-ZOPPA}
    Let \(f:\mathbb{R}^d\to 
\R\) satisfy Assumption~\ref{ass:on_f} and such that $f^{\lambda,\delta}$ is convex and denote $X^{\star}_{\lambda,\delta}:=\argmin f^{\lambda,\delta}$. Let the sequence \(\left(N_k\right)_{k\in\mathbb{N}}\) be such that \(\sum_{k\in\mathbb{N}}N_k^{-1/2}<\infty\) and let  $\left(x^k\right)_{k\in\mathbb{N}}$ be generated by S-ZOPPA (Algorithm \ref{alg:szd}). Assume that  \(\left(x^k\right)_{k\in\mathbb{N}}\) is bounded, and that there exists \(\varepsilon>0\) such that for any \(k\in\mathbb{N}\), \(\frac{1}{N_k}\sum_{i=1}^{N_k}\exp(-f(y^k_i)/\delta)\ge\varepsilon\) where \(\left(y^k_i\right)_{i=1}^{N_k}\) are the i.i.d. samples from \(\mathcal{N}(x^k,\lambda\delta I)\) generating \(\szprox^{\delta,N_k}_{\lambda, f}(x^k)\).  
Then,      \(\sum_{k\in\mathbb{N}}\|\zprox^{\delta}_{\lambda, f}(x^k)-x^k\|^2<\infty\)
    with probability one. Furthermore, there exists a random variable $\bar{x}$ taking values in $X^{\star}_{\lambda,\delta}$ such that $$x^k \to \bar{x} \text{ with probability one.} $$
\end{theorem}

The assumptions of Theorem~\ref{thm:S-ZOPPA} on the sequence \(\left(x^k\right)_{k\in\mathbb{N}}\) are required in order to bound the variance of the estimator \(\szprox^{\delta,N_k}_{\lambda, f}\) (see Lemma~\ref{lem:boundedness} in Appendix  \ref{app:proof_S-ZOPPA}). We also provide in Proposition~\ref{prop:stability} of Appendix  \ref{app:proof_S-ZOPPA} an estimate in high probability for this event. The proof then follows from the application of \cite[Theorem~2.5]{combettes2015stochastic} and \cite[Theorem~4.1]{combettes2026}, see Appendix \ref{app:proof_S-ZOPPA}.
 
\section{Numerical experiments}\label{sec:experiments}

\subsection{Role of the stepsize $\lambda$}\label{sec:exp_lambdarole}
 
In Section \ref{sec:main_results} we discussed the advantages of both $\lambda$ small and $\lambda$ large. So, how to choose? We recommend a continuation strategy in which $\lambda$ is initialized at a relatively large value and then decreased every $\bar{k}$ iterations or when the iterates are stationary enough. In this way the algorithm first explores a smoother and more convex (see Theorem~\ref{thm:convexification}) landscape and progressively approaches the geometry of the original objective $f$. In light of the stationarity results in Theorem~\ref{thm:rates}, this suggests that the iterates move from stationary points of $f^{\lambda,\delta}$ toward stationary points of $f$. We do not analyze the convergence of the method with a varying sequence $(\lambda_k)_k$ here, as this falls outside the scope of the present paper and is left for future work. Nevertheless, we provide numerical evidence showing that decreasing $\lambda$ eventually leads the algorithm to stationary points of $f$ and helps escaping from poor stationary points. We show this in Figure~\ref{fig:varying_lambda} and report the experiment setting in Appendix \ref{app:exp_setting_lambda}.

\begin{figure}[htbp]
    \centering
\includegraphics[width=1\linewidth]{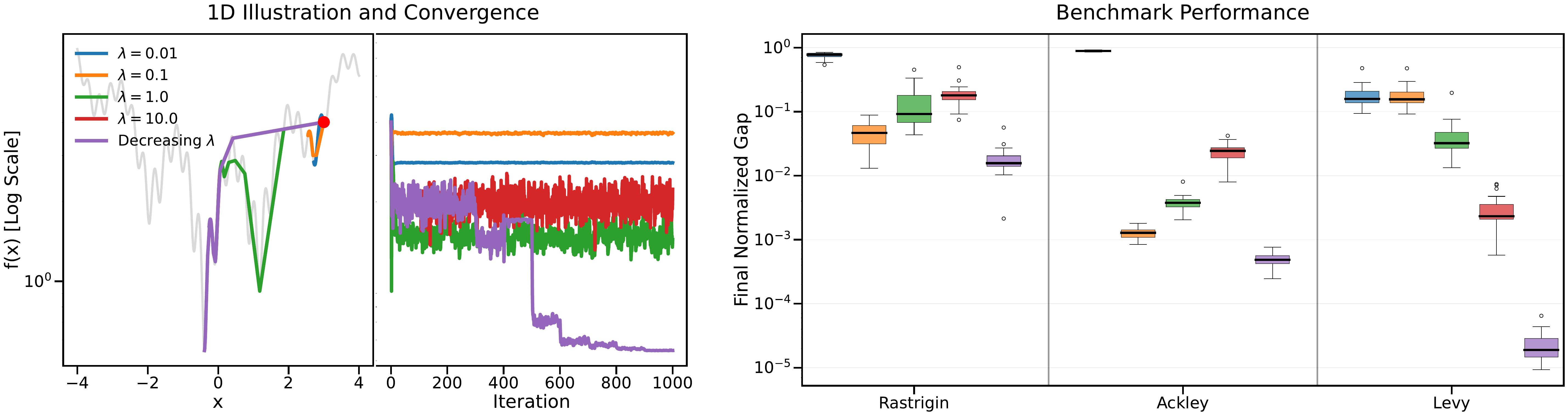}
    \caption{Comparison between fixed $\lambda$ and decreasing $\lambda$ with $\delta =1 >0 $ fixed.}
    \label{fig:varying_lambda}
\end{figure}

\subsection{Role of the temperature parameter $\delta$ and sampling considerations}\label{sec:sampling_challenges}

Estimating \(\zprox^{\delta}_{\lambda, f}\) accurately is crucial to keep the properties demonstrated on ZOPPA in the previous sections. We discuss below the properties of \(\szprox^{\delta,N}_{\lambda, f}\).

\textbf{S-ZOPO as an importance sampling estimator.} Interestingly, S-ZOPO can be seen as a Self Normalized Importance Sampling (SNIS) estimator \cite{deligiannidis2024importance} of \[\mathbb{E}_{y\sim\mu_x}[y],\quad \mu_x(y)\propto\exp(-f(y)/\delta)\exp(-\frac{1}{2\lambda\delta}\|x-y\|^2),\]
with the proposal density \(\mathcal{N}(x,\lambda\delta I)\). This class of estimators is well known in the literature and was previously referred to as ratio estimate \cite{hesterberg1988advances}, weighted importance sampling estimator \cite{lemieux2009monte, rubinstein2016simulation} or autonormalized importance sampling estimator \cite{agapiou2017importance}. In particular, as shown in \cite[Section~2.5.2]{hesterberg1988advances}, \(\szprox^{\delta,N}_{\lambda, f}\) is a biased, but consistent, estimate of \(\zprox^{\delta}_{\lambda, f}\) and, as long as \(f\) is lower bounded it has the following asymptotic behavior
\begin{equation}\begin{aligned}
\mathbb{E}&\left[\szprox^{\delta,N}_{\lambda, f}(x)\right]=\zprox^{\delta}_{\lambda, f}(x)+\frac{C_1}{N}+o(N^{-1}),\\
\mathbb{E}&\left\|\szprox^{\delta,N}_{\lambda, f}(x)-\zprox^{\delta}_{\lambda, f}(x)\right\|^2=\frac{C_2}{N}+o(N^{-1}),
\end{aligned}
\end{equation}
where \(C_1\) and \(C_2\) depend on \(x\), \(\lambda\) and \(\delta\).

However, these asymptotic rates on the bias and the variance of \(\szprox^{\delta,N}_{\lambda, f}\) are not informative of the practical behavior of the estimator in certain regimes, especially when \(\delta\) tends to \(0\) (a simple example is given in Appendix \ref{app:asymptotic_behavior_x2}).

\textbf{The role of the Effective Sample Size.} This phenomenon is connected to the notion of \emph{Effective Sample Size} (ESS), which is a common metric in importance sampling. Introduced in \cite{kish1965survey}, the ESS serves as a diagnostic tool to estimate the number of independent samples from the target distribution that would be required to achieve a variance equivalent to that of the \(N\) weighted samples from the proposal distribution \cite{kong1992note,liu1996metropolized}. It takes value from \(1\) (when one sample takes all the weight) to \(N\) (when all the samples have equal weight), and can be formulated as
\begin{equation}
    \widehat{ESS} = \frac{\left(\sum_{i=1}^N w_i \right)^2}{\sum_{i=1}^N w_i^2},
\end{equation}
where in the case of S-ZOPO, \(w_i=\exp\left(-\frac{f(y_i)}{\delta}\right)\).

The ESS is a relevant metric in the context of S-ZOPO, as it allows detecting when few samples determine the final value of the estimator, ultimately altering the behavior of the estimator. Suppose that the ESS is close to \(1\). Then, denoting the samples \(y_i=x+\sqrt{\lambda\delta}\varepsilon_i\), where \( \varepsilon_i\sim\mathcal{N}\left(0,I\right)\)
and supposing that the dominant weight is the first one, we get that
\[\szprox^{\delta,N}_{\lambda, f}(x)=\frac{\sum_{i=1}^N y_i\,
\exp\!\left(-\frac{f(y_i)}{\delta}\right)\,
}
{\sum_{i=1}^N
\exp\!\left(-\frac{f(y_i)}{\delta}\right)\,}\approx \arg\min_{\left(y_i\right)_{i=1}^N}f(y_i)=x+\sqrt{\lambda\delta}\varepsilon_1.\]
In other words, the weight collapse relative to the ESS tending to \(1\) implies that \(\szprox^{\delta,N}_{\lambda, f}(x)\) rather behaves as a \emph{Random search iteration} than as an estimator of \(\zprox^{\delta}_{\lambda, f}(x)\).

\textbf{The problem with the Random search regime.} While for a well suited Gaussian distribution and a good initialization point, Random search algorithm might be useful, reaching the Random search regime generates several issues. First, by entering this regime, one \textbf{loses any interpretation of \(\szprox^{\delta,N}_{\lambda, f}(x)\) as an approximation of the zeroth order proximal operator \(\zprox^{\delta}_{\lambda, f}(x)\)}. As a consequence, the results stated in the previous sections, such as the interpretation of ZOPO as the proximal operator of another potential, cannot be rigorously translated here. Second, if the scaling of the Gaussian distribution \(\sqrt{\lambda\delta}\) is too small with respect to the distance of \(x\) to a critical point \(\bar x\) of \(f\), the point \(\szprox^{\delta,N}_{\lambda, f}(x)\) will almost not move from \(x\) and the cost in terms of number of samples in order to make a \textit{significant step} scales exponentially. Indeed, Chernoff inequality allows us to write that for samples \(\left(y_i\right)_{i=1}^N\) from \(\mathcal{N}(x,\lambda\delta I)\), the probability to escape the ball of radius \(\sqrt{R}\) around \(x\) where \(R>d\lambda\delta\), satisfies
\begin{equation}\label{eq:escaping_probability}\mathbb{P}\left(\max_{i\in[N]}\|y_i-x\|^2\ge R\right)\le N\exp\left(-\frac{(R-d\lambda\delta)^2}{4R\lambda\delta}\right)=N\exp\left(-\frac{(\alpha-1)^2d}{4\alpha}\right),\end{equation}
where \(\alpha=\frac{R}{\lambda\delta d}\). Hence, for large \(\alpha\), the amount of samples \(N\) required to escape the ball scales with \(\exp(\alpha d)\).

\textbf{When does the ESS collapse?} The weight collapse for SNIS estimators is a known phenomenon and can be related to the curse of dimensionality \cite{bengtsson2008curse}. In our context, an important property is that the ESS goes to \(1\) when \(\delta\) tends to \(0\), and, for a fixed set of samples \(\left(y_i\right)_{i=1}^N\), the ESS is increasing in \(\delta\). This phenomenon is illustrated in the left graph of Figure~\ref{fig:bias+var_vs_N} (the experiment setting is in Appendix \ref{app:exp_setting_delta}).

\begin{figure}[htbp]
    \centering
    \includegraphics[width=0.32\linewidth]{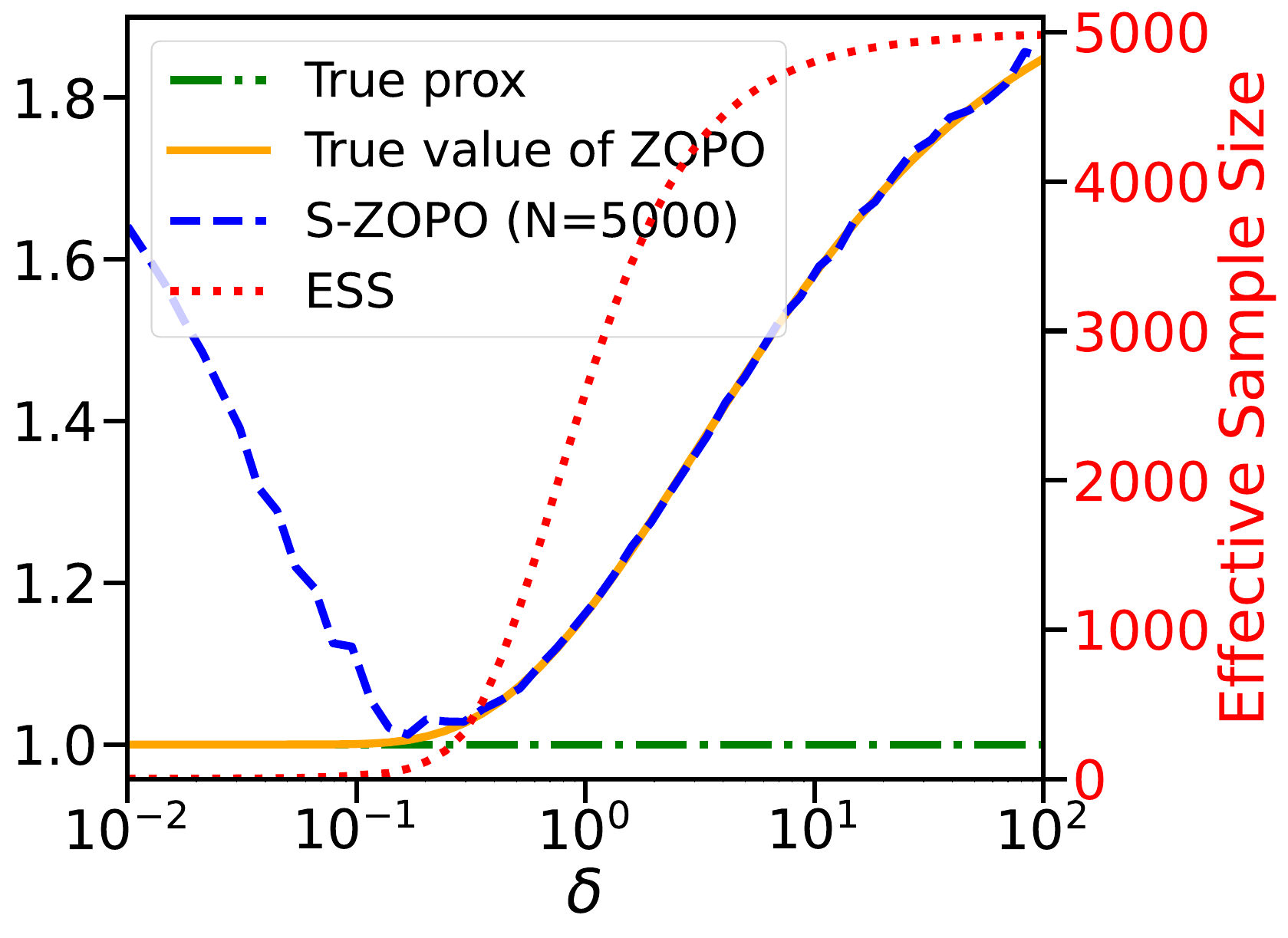}
    \hfill
    \includegraphics[width=0.32\linewidth]{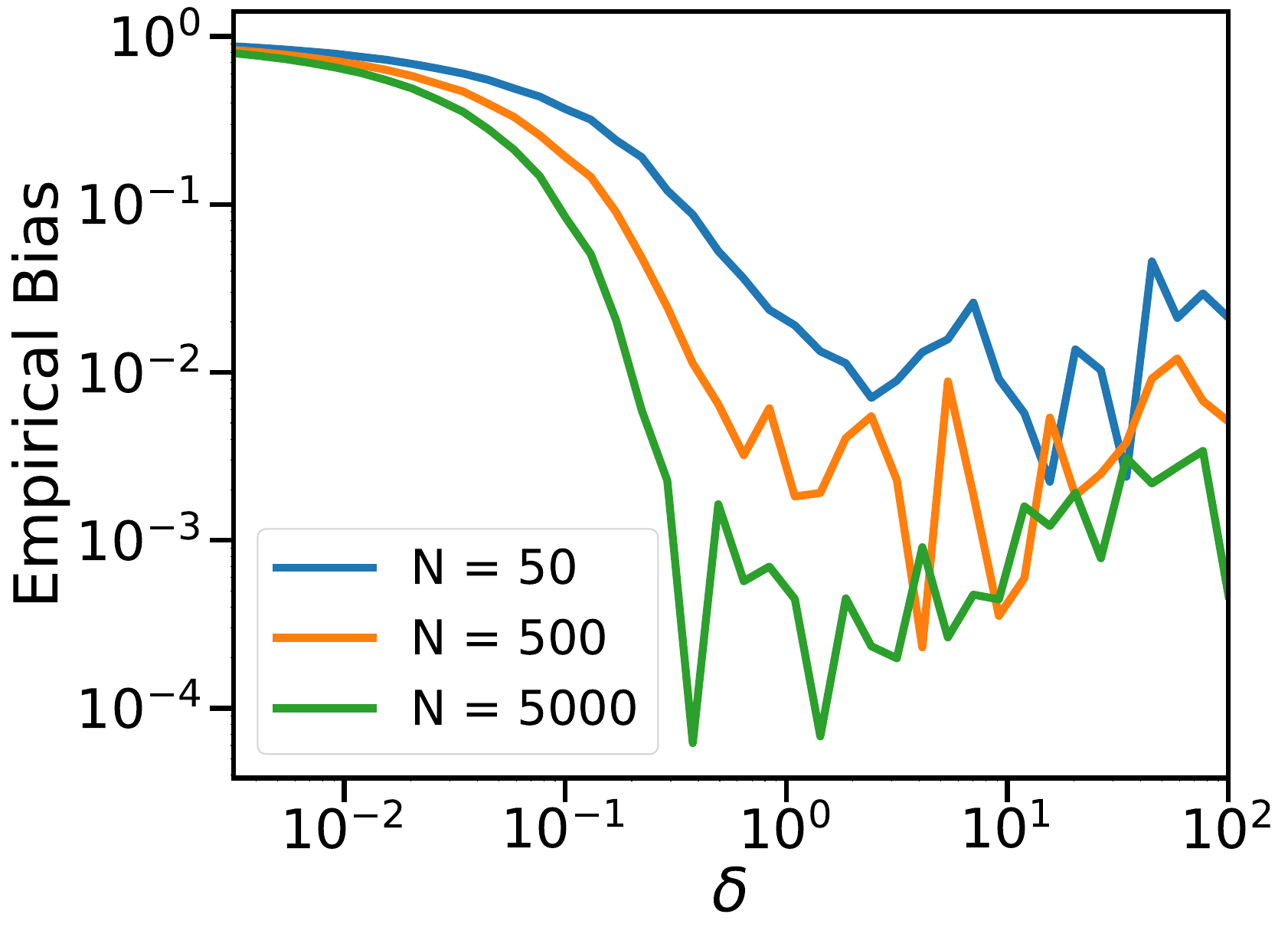}\hfill
    \includegraphics[width=0.32\linewidth]{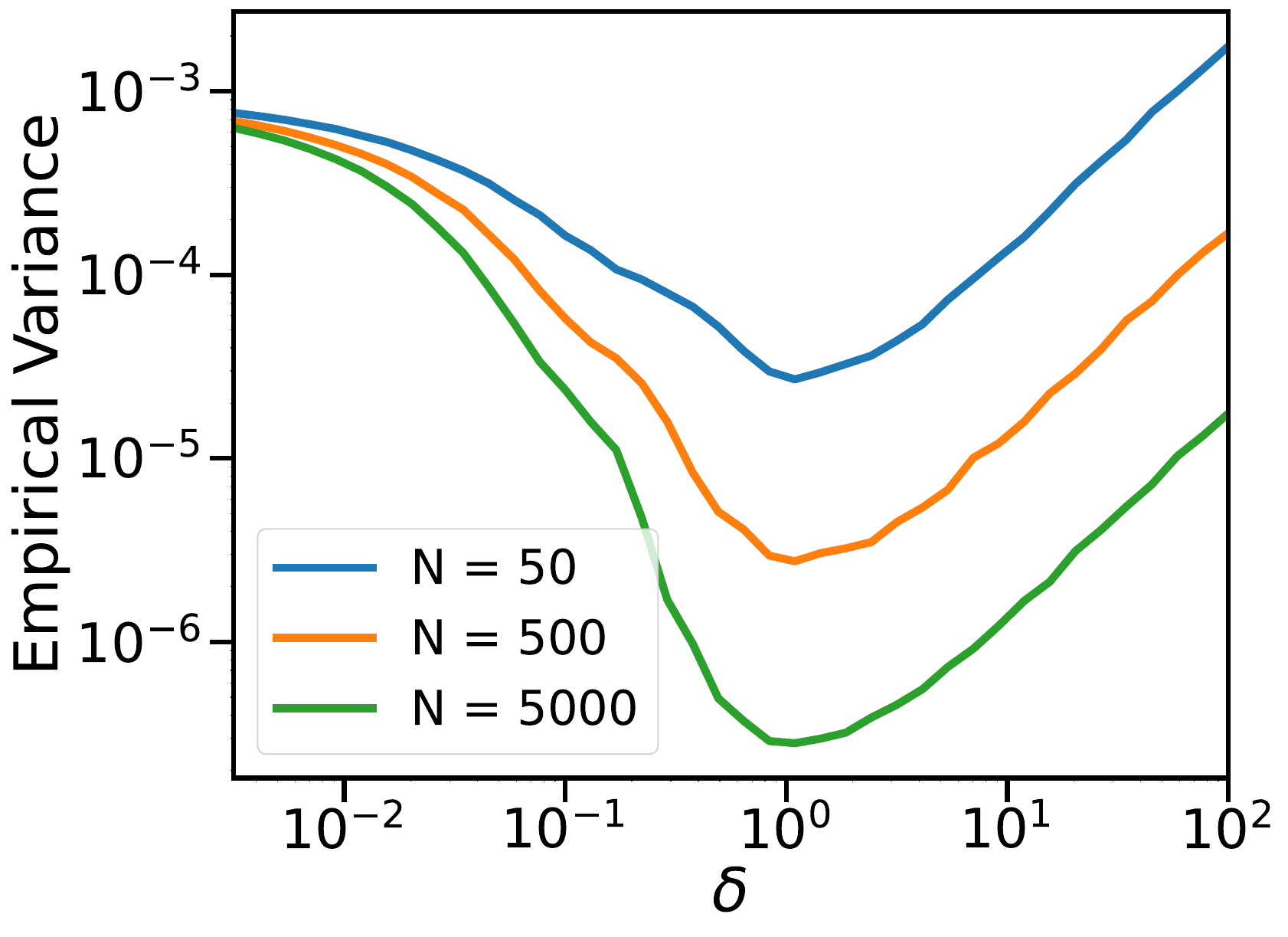}
    \caption{Effect of the temperature \(\delta\) on sampling for \(f:x\mapsto |x|\), with \(x=2\) and \(\lambda=1\). Left: ZOPO, S-ZOPO, classical prox, and mean ESS. Center and right: empirical bias and variance of S-ZOPO.}
    \label{fig:bias+var_vs_N}
\end{figure}

Such a behavior suggests that imposing \(\delta\) to go to \(0\) automatically brings S-ZOPO to the Random search regime. This proves to be critical, since a small $\delta$ causes the escaping probability \eqref{eq:escaping_probability} to vanish.
It is worth stressing that the bias and the variance of the estimator in the Random Search regime cannot be easily reduced by increasing the number of samples in Figure~\ref{fig:bias+var_vs_N}. It is indeed clear that when \(\delta\) is small, increasing the sample size is significantly less efficient to reduce both bias and variance compared to the high temperature regime (\(\delta\) large).

\section{Conclusion}\label{sec:conclusion}

We investigated zeroth-order proximal algorithms in the high-temperature regime, a setting that deviates from the traditional vanishing-temperature limit and necessitated an entirely new analytical framework. We showed that by keeping the temperature bounded away from zero, ZOPPA can be interpreted as an exact proximal-point method applied to an auxiliary objective, yielding convergence guarantees under extremely mild assumptions without requiring a cooling schedule. Our analysis of the sampled method S-ZOPPA further shows that this high-temperature regime is the natural setting for reliable, sample-based estimation.

Several directions remain open. A limitation of our work is that the convergence analysis does not yet extend straightforwardly to varying sequences \((\lambda_k)_k\) or
\((\delta_k)_k\); so this is left for future work. We would also like to remove the boundedness assumption in the convergence analysis of S-ZOPPA. On the other hand, this work paves the way for several new directions. A first direction is the development of zeroth-order splitting algorithms, where ZOPO could play the role of classical proximal operators in black-box settings. This would allow one to study zeroth-order analogues of several classical algorithms \cite{Rockafellar_FB,LionsMercier1979,Chambolle2011}, together with their acceleration counterparts \cite{FISTA, ChambolleDossal} and their generalizations \cite{BrediesChencheneNaldi2024,AkermanChencheneGiselssonNaldi2025}. Another direction is the extension to extended-valued functions \(f:\mathbb{R}^d\to\mathbb{R}\cup\{+\infty\}\) providing a natural way to handle constrained black-box optimization. Finally, our results motivate the use of alternative sampling strategies for estimating ZOPO. In the high-temperature regime, the relevant distributions are smoother and less concentrated, making techniques such as importance sampling, Langevin-type methods or sequential Monte Carlo more amenable to implementation and analysis.

\section*{Acknowledgments}
E. N., C. M. and S. V. acknowledge the support of the US Air Force Office of Scientific Research (FA8655-22-1-7034). H. L. acknowledge the financial support of the Ministry of Education, University and Research (FARE grant ML4IP R205T7J2KP). The research by E. N., C. M. and S. V. has been supported by the MUR Excellence Department Project awarded to Dipartimento di Matematica, Universita di Genova, CUP D33C23001110001. E.N. and S. V. are members of the Gruppo Nazionale per l’Analisi Matematica, la Probabilità e le loro Applicazioni (GNAMPA) of the Istituto Nazionale di Alta Matematica (INdAM). This work represents only the view of the authors. The European Commission and the other organizations are not responsible for any use that may be made of the information it contains.

\bibliographystyle{abbrv}
\bibliography{NEURIPS_version/bibliography}


\appendix

\section{Preliminaries}

\paragraph{Notation.}

To help the reader we summarize in Table \ref{tab:notations} the notations used along the paper and in the appendix.

\begin{table}[ht]
\caption{Summary of Notations and Parameters}
\label{tab:notations}
\centering
\begin{tabular}{@{}ll@{}}
\toprule
\textbf{Symbol} & \textbf{Definition} \\ \midrule
$d$ & Dimension of the ambient space \(\mathbb{R}^d\) \\
$I$ & The identity matrix in \(\mathbb{R}^d\) \\
\(\mathcal N(m,\Sigma)\) & Gaussian distribution with mean \(m\) and covariance matrix \(\Sigma\) \\
\(\lambda\) & Proximal parameter or stepsize \\
\(\delta\) & Temperature/viscosity parameter \\
\(\varphi_x(y)\) & \(f(y) + \frac{\|x-y\|^2}{2\lambda}\) \\
\(Z_{\lambda,\delta}(x)\) & Normalization constant: \(\int_{\mathbb{R}^d} \exp(-\varphi_x(y)/\delta)\,dy\) \\
\(q_x\) & Gibbs density of \(-\varphi_x/\delta\): \(\exp(-\varphi_x(y)/\delta)/Z_{\lambda,\delta}(x)\) \\
\(\mu_x\) & Probability measure with density \(q_x\) \\
\(T_{\lambda,\delta}\) & Deterministic ZOPO map: \(x\mapsto \zprox^\delta_{\lambda,f}(x)\) \\
\(f^\lambda\) & Moreau envelope: \(f^\lambda(x)=\inf_y\left\{f(y)+\frac{1}{2\lambda}\|y-x\|^2\right\}\) \\
\(f^{\lambda,\delta}\) & Soft Moreau envelope: \(f^{\lambda,\delta}(x)=-\delta\log\left((2\pi\lambda\delta)^{-d/2}Z_{\lambda,\delta}(x)\right)\) \\
\(\prox_{\lambda f}\) & Proximity operator: \(\prox_{\lambda f}(x)=\argmin_y\left\{f(y)+\frac{1}{2\lambda}\|x-y\|^2\right\}\) \\
\(X^\star\) & Set of minimizers of \(f\): \(X^\star=\argmin f\) \\
\(X^\star_{\lambda,\delta}\) & Set of minimizers of \(f^{\lambda,\delta}\): \(X^\star_{\lambda,\delta}=\argmin f^{\lambda,\delta}\) \\
\(\zprox^\delta_{\lambda,f}\) & Zeroth-order prox: 
\(\displaystyle
\zprox^\delta_{\lambda,f}(x)
=
\frac{
\mathbb{E}_{y\sim\mathcal{N}(x,\lambda\delta I)}
\left[y\exp(-f(y)/\delta)\right]
}{
\mathbb{E}_{y\sim\mathcal{N}(x,\lambda\delta I)}
\left[\exp(-f(y)/\delta)\right]
}
\) \\
\(\szprox^{\delta,N}_{\lambda,f}\) & Sampled zeroth-order prox:
\(\displaystyle
\frac{
\sum_{i=1}^N y_i \exp(-f(y_i)/\delta)
}{
\sum_{i=1}^N \exp(-f(y_i)/\delta)
},\quad
y_i\sim\mathcal{N}(x^k,\lambda\delta I)
\) \\
\(\operatorname{osc}(V)\) & Oscillation of \(V:\R^d\to \R\): \(\sup V-\inf V\) \\
\(\operatorname{dist}(x,A)\) & Euclidean distance of a point \(x\in \R^d\) to a set \(A\subset \R^d\) \\
\bottomrule
\end{tabular}
\end{table}

\subsection{Approximating the proximal operator with Hamilton-Jacobi}\label{app:HamiltonJacobi}

Consider a function \(f:\R^d\to\R\). For any \(\lambda>0\), the Moreau envelope of \(f\) is defined by
\[
f^\lambda(x)=\inf_{y\in\mathbb R^d}
\left\{
f(y)+\frac{1}{2\lambda}\|y-x\|^2
\right\}.
\]
The set of minimizers in this problem is denoted by \(\prox_{\lambda f}(x)\). When the minimizer is unique, we call it the proximal point, and \(\prox_{\lambda f}\) is the proximal point operator. In particular, when the proximal point is unique and the Moreau envelope is differentiable, for any \(x\in\R^d\) it holds that
\begin{equation}\label{eq:prox_env}\prox_{\lambda f}(x)=x-\lambda\nabla f^\lambda(x).\end{equation}

This property of the proximal operator is noticeable since it relates it to the Hamilton-Jacobi partial differential equation (we refer to \cite[Chapter 10]{evans} for the definition and properties of such PDE). It is indeed well known that the following system
\begin{equation}\label{eq:HJ_intro}\tag{HJ}
    \left\{\begin{aligned}
        &\partial_t u(x,t) + \frac{1}{2}\left\|\nabla u(x,t)\right\|^2 = 0,\\
        &u(x,0)=f(x),
        \end{aligned}\right.
\end{equation}
is solved by the Moreau envelope i.e. for any \(x\in\R^d\) and \(t\ge0\), \[u(x,t)=f^t(x).\]

By incorporating an additive term to \eqref{eq:HJ_intro}, one can obtain the viscous Hamilton-Jacobi or Hamilton-Jacobi-Bellman equation, depending on the parameter \(\delta>0\):
\begin{equation}\label{eq:HJB_intro}\tag{HJB}
    \left\{\begin{aligned}
        &\partial_t u^\delta(x,t) + \frac{1}{2}\left\|\nabla u^\delta(x,t)\right\|^2 = \frac{\delta}{2}\Delta u^\delta,\\
        &u^\delta(x,0)=f(x).
        \end{aligned}\right.
\end{equation}
By exploiting the Hopf-Lax transformation, i.e. introducing the variable \(v^\delta:x,t\mapsto \exp\left(-\frac{u^\delta(x,t)}{\delta}\right)\), one can observe that \(v^\delta\) is the solution of the heat equation with the following initial condition:
\[v^\delta(x,0)=\exp\left(-\frac{f(x)}{\delta}\right).\]
It follows that the soft Moreau envelope \(f^{\lambda,\delta}:=u^\delta(\cdot,\lambda)\) can be written as
\[\forall x\in\R^d,\quad f^{\lambda,\delta}(x)=-\delta\log\left(\left(2\pi \lambda\delta\right)^{-\frac{d}{2}}\int_{\R^d}\exp\left(-\frac{f(y)}{\delta}\right)\exp\left(-\frac{\|x-y\|^2}{2\lambda\delta}\right)dy\right).\]

The viscous proximal operator or Zeroth Order Proximal Operator is defined by analogy with \eqref{eq:prox_env} as
\begin{equation}\label{eq:zprox_softenvelope_app}
    \forall x \in\R^d,\quad \zprox^{\delta}_{\lambda, f}(x)=x-\lambda\nabla f^{\lambda,\delta}(x).
\end{equation}

It is easy to see by simple computations, once noticed that derivation under the integral sign is possible, that the latter corresponds to \eqref{eq:def_zprox_intro}. We formalize this observation in the following remark, since we will use it several times later on.

\begin{remark}[Differentiation under the integral sign]\label{rem:differentiation}
The derivations under the integral sign and the integration by parts in our results are justified. Assuming $f$ is bounded from below, the function $\exp\left(-\frac{f(y)}{\delta}\right)\exp\left(-\frac{\|x-y\|^2}{2\lambda\delta}\right)$ exhibits Gaussian tails due to the $\exp(-\|x-y\|^2/(2\lambda\delta))$ factor. Consequently, for any $x$ restricted to a compact neighborhood, the derivatives of the integrand with respect to $x$ are uniformly bounded by a function of the form $P(\|y\|)\exp(-c\|y\|^2)$, where $P$ is a polynomial and $c>0$. Because this dominating function is absolutely integrable on $\mathbb{R}^d$, the Dominated Convergence Theorem guarantees the existence of all moments (e.g., $\int_{\mathbb{R}^d}\|y\|^2 q_x(y) dy < \infty$) and permits interchanging the gradients $\nabla_x$ and $\nabla_x^2$ with the integral over $y$.
\end{remark}

\subsection{Some other known results}\label{app:known_results}

The operator \(\zprox_{\lambda,f}^{\delta}\) has appeared in several closely related forms in the recent literature, mostly as a smooth, zeroth-order approximation of the classical proximal operator. In \cite{osher2023hamilton}, the formula is derived from the viscous Hamilton--Jacobi equation through the Cole--Hopf transformation, as recalled in Appendix~\ref{app:HamiltonJacobi}. The resulting object, called HJ-Prox, is proposed as a way to approximate proximal operators using only function evaluations, possibly from black-box oracles. Under weak convexity assumptions, and for a sufficiently small proximal time so that the proximal point is unique, \cite[Theorem 1]{osher2023hamilton} proves that this approximation converges to the exact proximal operator as \(\delta\to0\). The same work also emphasizes the smoothing effect of the viscosity parameter \(\delta\).

A complementary viewpoint is developed in \cite{tibshirani2025laplace}, where the same formula is obtained directly from Laplace's method. More generally, for an infimal convolution
\[
    \inf_y \{f(y)+g(x-y)\},
\]
the authors consider the self-normalized Laplace approximation
\[
    y_x^\delta
    =
    \frac{
    \int y\exp\left(-\frac{f(y)+g(x-y)}{\delta}\right)\,dy
    }{
    \int \exp\left(-\frac{f(y)+g(x-y)}{\delta}\right)\,dy
    }.
\]
Choosing \(g(z)=\|z\|^2/(2\lambda)\) gives exactly \(\zprox_{\lambda,f}^{\delta}\). This perspective shows that ZOPO is a special case of a broader Laplace approximation to minimizers of infimal-convolution problems. The analysis in \cite{tibshirani2025laplace} proves convergence of \(y_x^\delta\) to the corresponding minimizer as \(\delta\to0\), under mild assumptions, and also hints to extensions such as smoothed projections and Bregman-type proximal maps.

Related results were obtained in \cite{zhang2024inexact}. There, the Gibbs expectation associated with
\[
    \varphi_x(y)=f(y)+\frac{1}{2\lambda}\|y-x\|^2
\]
is studied as an approximation of the exact proximal point. If \(\varphi_x\) has a unique global minimizer, this expectation converges to \(\prox_{\lambda f}(x)\) as \(\delta\to0\); under additional nondegeneracy and local \(C^2\) assumptions, the approximation error is of order \(O(\delta)\), see \cite[Theorem 4]{zhang2024inexact}. In the multi-valued case, the limiting point belongs to the convex hull of the proximal set \cite[Corollary 1]{zhang2024inexact}. These approximation results are then used within an inexact proximal-point framework for zeroth-order global optimization, including Monte Carlo and tensor-train based estimators.

Finally, ZOPPA was also introduced as Consensus Hopping (CH) in \cite{Fornasier2024}. A convergence analysis restricted to the small-\(\delta\), large-\(\lambda\) regime was recently given in \cite{Fornasier2026}. In particular, under local Lipschitzianity and quadratic growth assumptions on \(f\), it is shown that, for \(\delta\) sufficiently small and \(\lambda\) sufficiently large, the Gaussian measure
\[
    \rho_k:=\mathcal{N}(x^k,\lambda\delta I)
\]
associated with the CH iteration converges to a stationary distribution.

Overall, these works exploit the regime \(\delta\to0\), where \(\zprox_{\lambda,f}^{\delta}\) approximates the true proximal operator and can therefore be treated as an inexact proximal map. In contrast, the present paper studies the fixed-temperature regime. Rather than viewing ZOPO as an approximation error around \(\prox_{\lambda f}\), we show that for every fixed \(\delta>0\), ZOPPA is an exact proximal-point method for an auxiliary objective and derive convergence guarantees directly in this high-temperature regime.

\section{Main convergence results: proofs for Sections \ref{sec:setting} and \ref{sec:main_results}}

We aim to study the following method
\begin{equation}\label{eq:zprox_alg}
    x^0\in \R^d, \qquad x^{k+1} = \zprox_{\lambda, f}^\delta (x^k),\quad k \geq 0.
    \end{equation}

While the existent literature focus on the regime $\delta \to 0$ and relies on the fact that $\zprox _{\lambda, f}^\delta \to \prox_{\lambda f}$ for $\delta\to 0$, we focus on the case $\delta$ fixed and bounded away from zero. While in \cite{zhang2024inexact} interpret the algorithm as an inexact proximal point algorithm, we cannot interpret the method as such. In fact, through this view point we can only say that we commit an error at each iteration which size does not decrease along iterations, and this a priori could hinge the convergence. However, we show in this section that the convergence of the method can be still guaranteed when $\delta$ is fixed.

\subsection{Interpretation as MMSE}

In this section we interpret the zeroth order prox as an MMSE. By doing this, we enable ourselves to exploit the part of literature that provides convergence guarantees for algorithms using MMSE in PnP methods \cite{Gribonval11,Gribonval2013,xu-2020-provable}.

Let $f:\mathbb{R}^d\to\mathbb{R}$ satisfying Assumption~\ref{ass:on_f} and fix $\delta>0$ and $\lambda>0$.
Define the Gibbs density
\begin{equation}\label{eq:Cdelta}
    p_Y(y)=\frac{\exp\!\left(-\frac{f(y)}{\delta}\right)}{C_{\delta}}, \qquad C_{\delta}:= \int \exp\left(\frac{-f(y)}{\delta}\right)\, dy,
\end{equation}
and consider the additive Gaussian noise model
\[
X \;=\; Y + B,\qquad B \sim\mathcal N(0,\lambda\delta \, \mathrm{Id}),
\]
where $Y\sim p_Y$ and $B$ is independent of $Y$. Clearly, this implies that
\begin{equation}
p_X(x) =\frac{1}{C_\delta (2\pi\lambda \delta)^{d/2}}\int_{\mathbb{R}^d}\exp\!\left(-\frac{f(y)}{\delta}\right)\,
\exp\!\left(-\frac{\|y-x\|^2}{2\lambda\delta}\right)\,dy.
\end{equation}

The likelihood of $X$ given $Y=y$ is
\[
p_{X\mid Y}(x\mid y)
=\frac{1}{(2\pi\lambda\delta)^{d/2}}
\exp\!\left(-\frac{\|x-y\|^2}{2\lambda\delta}\right).
\]
Hence Bayes' rule yields the posterior (up to normalization):
\begin{align*}
p_{Y\mid X}(y\mid x)
&= \frac{p_{X\mid Y}(x\mid y)\,p_Y(y)}{p_X(x)} \\
&= \frac{1}{Z_{\lambda,\delta}(x)}\exp\!\left(-\frac{\|x-y\|^2}{2\lambda\delta}\right)\,
       \exp\!\left(-\frac{f(y)}{\delta}\right)\\
&=\frac{1}{Z_{\lambda,\delta}(x)}\exp\!\left(
-\frac{1}{\delta}\Big(f(y)+\frac{1}{2\lambda}\|y-x\|^2\Big)
\right),
\end{align*}
where
\begin{equation}
Z_{\lambda,\delta}(x)\;:= C_\delta (2\pi\lambda \delta)^{d/2} p_X(x) =\;\int_{\mathbb{R}^d}\exp\!\left(-\frac{f(y)}{\delta}\right)\,
\exp\!\left(-\frac{\|y-x\|^2}{2\lambda\delta}\right)\,dy.
\label{eq:Zz}
\end{equation}
The minimum mean-square error (MMSE) estimator of $Y$ given $X=x$ is the posterior mean:
\begin{align*}
\mathbb{E}[Y\mid X=x]
&=\int_{\mathbb{R}^d} y\,p_{Y\mid X}(y\mid x)\,dy \\
&=\frac{\int_{\mathbb{R}^d} y\,\exp\!\left(-\frac{f(y)}{\delta}\right)\exp\!\left(-\frac{\|y-x\|^2}{2\lambda\delta}\right)\,dy}
        {\int_{\mathbb{R}^d}   \exp\!\left(-\frac{f(y)}{\delta}\right)\exp\!\left(-\frac{\|y-x\|^2}{2\lambda\delta}\right)\,dy}.
\end{align*}
Therefore the zeroth-order proximal operator (ZOPO) is exactly the MMSE denoiser:
\[
\zprox_{\lambda, f}^{\delta}(x)\;=\;\mathbb{E}[Y\mid X=x].
\]
The interpretation of ZOPO as MMSE is related, but not equivalent, to the analysis carried out in \cite{darbon2021connecting}.

\subsection{Zeroth order prox as prox of another function: proof of Theorem~\ref{thm:zoppa_is_prox}}\label{app:proof_of_zoppa_is_prox}

In this section, we exploit the relation between MMSE and proximity operators first discovered in \cite{Gribonval11} and later extended to non-white noise $B$ in \cite{Gribonval2013}. The first result we give is partly a consequence of these works, but we also establish the surjectivity of the zeroth-order proximal operator $\zprox_{\lambda,f}^\delta$. By considering strictly positive Gibbs-type measures of the form $e^{-f/\delta}$ satisfying Assumption~\ref{ass:on_f}, we demonstrate that the operator is not merely smooth, but a global diffeomorphism.

\begin{proposition}[Bijectivity and regularity of ZOPO]\label{prop:zopo_surjective}
Let \(f:\mathbb{R}^d\to\mathbb{R}\) satisfy Assumption~\ref{ass:on_f}. Then $\zprox^{\delta}_{\lambda, f}$ is a smooth diffeomorphism, i.e., it is $C^\infty$, bijective and with $C^\infty$ inverse.
\end{proposition}
\begin{proof}  First, according to \cite[Corollary 1 (1)]{Gribonval2013}, the map $T_{\lambda,\delta}:= \operatorname{zprox}_{\lambda,f}^\delta$ is one-to-one (injective), which ensures that the inverse $T_{\lambda,\delta}^{-1}: \mathrm{Im}(T_{\lambda,\delta}) \to \mathbb{R}^d$ is well defined. Furthermore, \cite[Corollary 1 (2)]{Gribonval2013} establishes that $T_{\lambda,\delta}$ is $C^\infty$ and that its inverse is also $C^\infty$ on its image. 

It remains to show that $\mathrm{Im}(T_{\lambda,\delta})= \R^d$. Define the measure
\[
    d\nu(y)
    :=
    \exp\left(
        -\frac{f(y)}{\delta}
        -
        \frac{\|y\|^2}{2\lambda\delta}
    \right)\,dy.
\]
Clearly, it is finite ($\nu(\mathbb{R}^d)\leq C_\delta$) and it is strictly positive, so that $\operatorname{supp}(\nu)=\mathbb{R}^d$. Notice that for every
\(\theta,y\in\mathbb{R}^d\), we have $\langle \theta,y\rangle-\frac{\|y\|^2}{2\lambda\delta}\leq\frac{\lambda\delta}{2}\|\theta\|^2 $ and thus
\[
    \int_{\mathbb{R}^d}
    e^{\langle \theta,y\rangle}\,d\nu(y)
    \leq
    \exp\left(\frac{\lambda\delta}{2}\|\theta\|^2\right)
    \int_{\mathbb{R}^d} e^{-f(y)/\delta}\,dy
    <+\infty .
\]
Therefore, we can consider the exponential family of measures \cite[Equation (3.5)]{WainwrightJordan2008} \[\left\{p_\theta(y) = \frac{e^{\langle\theta,y\rangle}}{\int_{\mathbb{R}^d}
    e^{\langle \theta,z\rangle}\,d\nu(z)} \mid \theta \in \R^d\right\}.\] 
    Notice that here the sufficient statistic (see \cite[Section 3.2]{WainwrightJordan2008} for details) is the identity and that the family is minimal (see definition in \cite[Section 3.2]{WainwrightJordan2008}), because if $\langle a,y\rangle=b$ $\nu$-a.e. for some \(a\in\mathbb{R}^d\) and \(b\in\mathbb{R}\), then the support of \(\nu\) would be contained in an affine hyperplane unless \(a=0\). Since \(\operatorname{supp}(\nu)=\mathbb{R}^d\), this is impossible for \(a\neq0\). Consider now the log-partition function \cite[Equation (3.6)]{WainwrightJordan2008}
\[
    A(\theta)
    :=
    \log
    \int_{\mathbb{R}^d}
    e^{\langle \theta,y\rangle}\,d\nu(y),
    \qquad \theta\in\mathbb{R}^d .
\]
Since the Gaussian term gives quadratic decay, differentiation under the integral sign is justified,
and
\[
    \nabla A(\theta)
    =
    \frac{
    \displaystyle\int_{\mathbb{R}^d}
    y e^{\langle\theta,y\rangle}\,d\nu(y)
    }{
    \displaystyle\int_{\mathbb{R}^d}
    e^{\langle\theta,y\rangle}\,d\nu(y)
    }.
\]
Expanding the quadratic term we notice that $-\frac{\|y-x\|^2}{2\lambda\delta}= -\frac{\|y\|^2}{2\lambda\delta}+\frac{\langle x,y\rangle}{\lambda\delta}-\frac{\|x\|^2}{2\lambda\delta}$, and since the last term is independent of \(y\), it cancels in the ratio defining \(T_{\lambda,\delta}\). Hence
\[
    T_{\lambda,\delta}(x)
    =
    \nabla A\left(\frac{x}{\lambda\delta}\right).
\]

By definition, the mean-parameter space is \cite[Equation (3.26)]{WainwrightJordan2008}
\[M=
    \left\{
    m\in\mathbb{R}^d:
    \exists r\geq 0,\ \int r\,d\nu=1,\ 
    m=\int_{\mathbb{R}^d} y\,r(y)\,d\nu(y)
    \right\}.
\]
Let \(z\in\mathbb{R}^d\). Since \(\nu\) has a strictly positive Lebesgue density, say $d\nu(y)=h(y)\,dy$, $h(y)>0$, we can choose the gaussian
\[
    g_z(y)
    =
    (2\pi)^{-d/2}
    \exp\left(-\frac{\|y-z\|^2}{2}\right),
\]
which has mean \(z\). Define
\[
    r_z(y):=\frac{g_z(y)}{h(y)}.
\]
Then \(r_z\) is a density with respect to \(\nu\), since
\[
    \int_{\mathbb{R}^d} r_z(y)\,d\nu(y)
    =
    \int_{\mathbb{R}^d} g_z(y)\,dy
    =
    1,
\]
and its mean is
\[
    \int_{\mathbb{R}^d} y\,r_z(y)\,d\nu(y)
    =
    \int_{\mathbb{R}^d} y\,g_z(y)\,dy
    =
    z.
\]
Since \(z\in\mathbb{R}^d\) was arbitrary, \(M=\mathbb{R}^d\), and hence
the interior of $M$ is $\R^d$. We now apply the mean-parameterization theorem for minimal exponential families \cite[Theorem 3.3]{WainwrightJordan2008}: for a minimal exponential family, the gradient of the log-partition function maps the natural parameter space onto the interior of the mean-parameter space. In our case
\[
    \nabla A(\mathbb{R}^d)=\mathbb{R}^d .
\]
Finally, since \(x/(\lambda\delta)\) ranges over all of \(\mathbb{R}^d\) as
\(x\) ranges over \(\mathbb{R}^d\), we obtain
\[
    T_{\lambda,\delta}(\mathbb{R}^d)
    =
    \mathbb{R}^d .
\]
\end{proof}

\begin{theorem}
    Suppose Assumption~\ref{ass:on_f} is satisfied. Then, there exists a function $H_{\lambda,\delta}:\R^d \to \R$ such that
    \begin{equation}\label{eq:Zprox_is_prox}
    \zprox^{\delta}_{\lambda, f}= \prox_{\lambda \delta H_{\lambda,\delta}},
\end{equation}
and therefore the sequence $(x^k)_{k\in\mathbb{N}}$ generated by Algorithm \ref{alg:zoppa} satisfies
\begin{equation}\label{eq:Zprox_is_prox_alg}
    x^{k+1} = \prox_{\lambda \delta H_{\lambda,\delta}} (x^k).
\end{equation}
i.e., $\{x^k\}_k$ is a proximal point iteration.
\end{theorem}
\begin{proof} Let $T_{\lambda,\delta}:= \zprox_{\lambda,f}^\delta$ and let $\sigma^2 := \lambda \delta$. First, let define $\tilde X = \sigma^{-1} X$ so that \[\tilde X =\sigma^{-1}Y + \tilde B\quad \text{with} \quad \tilde B\sim \mathcal{N}(0,I),\] and the operator \[\tilde T(x):=\mathbb{E}[\sigma^{-1}Y \,|\, \tilde X = x].\]
Clearly, we have
\[\begin{aligned}
p_{\tilde X}(x) & = \frac{\sigma^d}{C_\delta (2\pi)^{d/2} }\int_{\mathbb{R}^d}\exp\!\left(-\frac{f(\sigma y)}{\delta}\right)\,
\exp\!\left(-\frac{\|x-y\|^2}{2}\right)\,dy \\ &=\frac{1}{C_\delta (2\pi)^{d/2} }\int_{\mathbb{R}^d}\exp\!\left(-\frac{f(y)}{\delta}\right)\,
\exp\!\left(-\frac{\|x-\sigma^{-1}y\|^2}{2}\right)\,dy = \sigma^{d}p_{X}(\sigma \, x).
\end{aligned}\]
By \cite[Theorem 1 (3)]{Gribonval2013} (which is \cite[Theorem 1 (2)]{Gribonval11}), we have
\[\tilde T = \prox_{\tilde H}\]
with 
\[\tilde H(x) := - \frac{1}{2} \|\tilde T ^{-1} (x) - x\|^2-\log p_{\tilde X} (\tilde T^{-1}(x)).\] 
By \cite[Corollary 1 (3)]{Gribonval2013}) we have
\[T_{\lambda,\delta}(x) = \argmin \frac{1}{2}\|y-x\|_{\sigma^2}^2 + \bar H(y) =\prox_{\sigma^2 \bar H}(x),\]
with 
\[\bar H(x) = \tilde H(\sigma^{-1}x) = - \frac{1}{2} \|\tilde T ^{-1} (\sigma^{-1} x) - \sigma^{-1}x\|^2-\log p_{\tilde X} (\tilde T^{-1}(\sigma^{-1}x)).\]
Now, it is possible to see (and it is explicit in the computations in \cite[Corollary 1]{Gribonval2013}) that $\sigma^{-1}   T_{\lambda,\delta} = \tilde T \circ \sigma^{-1}$. So that, we have
\[\begin{aligned}
    \bar H(x) &= - \frac{1}{2\sigma^2} \| T_{\lambda,\delta} ^{-1} (x) - x\|^2-\log p_{\tilde X} (\sigma^{-1}  T_{\lambda,\delta}^{-1}(x))\\
    & = - \frac{1}{2\sigma^2} \| T_{\lambda,\delta} ^{-1} (x) - x\|^2-\log p_{ X} ( T_{\lambda,\delta}^{-1}(x)) - d\log(\sigma).
    \end{aligned}\]
Recalling the definition of $Z_{\lambda,\delta}$ in \eqref{eq:Zz} and $C_\delta$ in \eqref{eq:Cdelta}, we have
\begin{equation}
- \log p_X(x) =-\log Z_{\lambda,\delta}(x) + \log C_\delta+\frac{d}{2}\log(2\pi\lambda\delta).
\label{eq:h_def_noD}
\end{equation}
and the fact that $\frac{1}{\delta}f^{\lambda,\delta}=- \log Z_{\lambda,\delta} +\frac{d}{2}\log(2\pi\lambda\delta)$, we  can define the function
\begin{equation}
H_{\lambda,\delta}(x) =
    -\frac{1}{2\lambda\delta}\big\| T_{\lambda,\delta}^{-1}(x)-x\big\|^2\;+\;\frac{1}{\delta}f^{\lambda,\delta}\!\big( T_{\lambda,\delta}^{-1}(x)\big).
\label{eq:H_def_app}
\end{equation} 
and have $\prox_{\lambda \delta H_{\lambda,\delta}} (x) = \prox_{\lambda \delta \bar H} (x) = T_{\lambda,\delta}(x)$.
\end{proof}

From this characterization, we can prove convergence of our method. We do this by first defining some properties of $H_{\lambda,\delta}$.

\begin{lemma}[Bounded below]
Under Assumption~\ref{ass:on_f}, the function $H_{\lambda,\delta}$ satisfies
\[H_{\lambda,\delta}(z)\geq - \log C_\delta +\frac{d}{2} \log(2\pi \lambda \delta), \qquad \forall z \in \R^d,\]
and therefore is bounded below.
\end{lemma}
\begin{proof}
Since $\exp(-\|z-y\|^2/(2\lambda\delta))\le 1$, for $Z_{\lambda,\delta}$ defined as in \eqref{eq:Zz}, we have $Z_{\lambda,\delta}(z)\le C_\delta$ for all $z$, hence
$-\log Z_{\lambda,\delta}(z)\ge -\log C_\delta$ and therefore
\[
\frac{1}{\delta}f^{\lambda,\delta}(z) = -\log Z_{\lambda,\delta}(z) + \frac{d}{2}\log(2\pi \lambda \delta) \ge - \log C_\delta + \frac{d}{2}\log(2\pi\lambda\delta)\qquad\forall z\in\mathbb{R}^d.
\]
Let $z \in \R^d$, then, by \eqref{eq:Zprox_is_prox} and \eqref{eq:H_def_app} we have
\[\begin{aligned}
    H_{\lambda,\delta}(z) &\geq H_{\lambda,\delta}(T_{\lambda,\delta}(z)) + \frac{1}{2\lambda \delta}\|T_{\lambda,\delta}(z) - z\|^2 \\
    &= \frac{1}{\delta} f^{\lambda, \delta}(z) -\frac{1}{2\lambda \delta}\|z - T_{\lambda,\delta}(z)\|^2 + \frac{1}{2\lambda \delta}\|T_{\lambda,\delta}(z) - z\|^2 \\
    & = \frac{1}{\delta}f^{\lambda, \delta}(z)\geq - \log C_\delta + \frac{d}{2} \log(2\pi \lambda \delta).  
    \end{aligned}\]
\end{proof}

\begin{proposition}\label{thm:zprox_convergence}
Let $f:\R^d\to \R$ satisfy Assumption~\ref{ass:on_f}. Then, the sequence $(x^k)_{k\in\mathbb{N}}$ generated by
Algorithm \ref{alg:zoppa} satisfies
\begin{enumerate}
\item[(i)] (\emph{Descent}) $H_{\lambda,\delta}(x^{k+1})\le H_{\lambda,\delta}(x^k) - \frac{1}{2\lambda \delta}\|x^{k+1}-x^k\|^2$ for all $k\ge 0$.
\item[(ii)] (\emph{Asymptotic stationarity}) $\|x^{k+1}-x^k\|\to 0$ as $k\to\infty$.
\end{enumerate}
\end{proposition}
\begin{proof}
    This can be seen as a specialization of the convergence result in \cite[Theorem~1]{xu-2020-provable} to the denoising-only iteration  (i.e.\ the case $g\equiv 0$, in the notation of \cite{xu-2020-provable}). By \eqref{eq:Zprox_is_prox}, we have immediately \emph{(i)}. By summing \emph{(i)} from $0$ to $K$ we obtain
    \[\sum_{k=1}^{K} \frac{1}{2\lambda \delta}\|x^{k+1}-x^k\|^2 \leq H_{\lambda,\delta} (x^0) - H_{\lambda,\delta}(x^{K+1}).\]
    Using the previous lemma, we obtain
    \[\sum_{k=1}^{K} \frac{1}{2\lambda \delta}\|x^{k+1}-x^k\|^2 \leq H_{\lambda,\delta} (x^0) - \inf H_{\lambda,\delta} <+\infty,\]
    and $\|x^{k+1}-x^k\|^2$ is summable and thus \emph{(ii)} holds.
\end{proof}

\subsection{Proof of Theorem~\ref{thm:zoppa_convergence}}\label{app:proof_of_zoppa_convergence}

Given the form of $H^{\lambda, \delta}$ defined in \eqref{eq:H_def_app}, we can derive descent also for the soft envelope.

\begin{theorem}\label{thm:descent_on_f}
Let $f$ satifying Assumption~\ref{ass:on_f}. Then, the sequence $(x^k)_{k\in\mathbb{N}}$ generated by
Algorithm \ref{alg:zoppa} satisfies
\begin{enumerate}
\item[(i)] (\emph{Descent}) $f^{\lambda,\delta}(x^{k})\le f^{\lambda,\delta}(x^{k-1}) - \frac{1}{2\lambda}\|x^{k}-x^{k-1}\|^2$ for all $k\ge 1$.
\item[(ii)] (\emph{Asymptotic stationarity}) $\|\nabla f^{\lambda, \delta}(x^k)\|=\frac{\|x^{k+1}-x^k\|}{\lambda}\to 0 , \qquad k\to\infty$.
\end{enumerate}
\end{theorem}
\begin{proof}
    From Theorem~\ref{thm:zprox_convergence} we have that for all $k\geq 1$
    \[H_{\lambda,\delta}(x^{k+1}) \leq H_{\lambda,\delta}(x^k) - \frac{1}{2\lambda\delta}\|x^{k+1}-x^k\|^2.\]
    On the other hand, for $k\geq 1$, we have $x^k= T_{\lambda,\delta}(x^{k-1})$ and by definition of $H_{\lambda,\delta}$,
    \[H_{\lambda,\delta}(x^{k})=-\frac{1}{2\lambda \delta} \|x^{k-1}-x^k\|^2 + \frac{1}{\delta}f^{\lambda, \delta}(x^{k-1}).\]
    Therefore, 
    \[\frac{1}{\delta}f^{\lambda,\delta}(x^{k}) -\frac{1}{2\lambda \delta}\|x^k-x^{k+1}\|^2 \leq \frac{1}{\delta} f^{\lambda,\delta}(x^{k-1}) -\frac{1}{2\lambda \delta}\|x^{k-1}-x^{k}\|^2-\frac{1}{2\lambda \delta}\|x^k-x^{k+1}\|^2,\]
    and the thesis.
\end{proof}

\begin{remark}[Stationary points of $f^{\lambda,\delta}$and of $Z_{\lambda,\delta}$]
From \eqref{eq:zprox_softenvelope_app} it is immediate to see that a point $x^\star$ is a fixed point of the zeroth order prox operator,
$x^\star=\zprox_{\lambda, f}^{\delta}(x^\star)$, if and only if
\begin{equation}
\nabla f^{\lambda,\delta}(x^\star)=0.
\label{eq:fixed_point_stationary}
\end{equation}
Moreover, since $Z_{\lambda,\delta}(x)>0$, we have $\nabla f^{\lambda,\delta}(x)=0 \iff \nabla \log Z_{\lambda,\delta}(x)=0 \iff \nabla Z_{\lambda, \delta}(x)=0$.
Thus, fixed points are precisely stationary points of $f^{\lambda,\delta}$ (equivalently, stationary
points of $Z_{\lambda,\delta}$).
\end{remark}

With the following result we provide convergence of the iterates and the rate on the residuals.

\begin{theorem}[Convergence of the iterates]
Let $f:\R^d\to \R$ satisfy Assumption~\ref{ass:on_f}. Then $f^{\lambda,\delta}$ admits at least a minimizer, satisfies the KL property and the sequence generated by Algorithm \ref{alg:zoppa} satisfies
\[\sum_k \|x^{k+1}-x^k\| < +\infty.\]
In particular  $\{x^k\}_k$  converges to some $x_{\lambda,\delta}^\star$ such that $x_{\lambda,\delta}^\star=\zprox_{\lambda, f}^{\delta}(x_{\lambda,\delta}^\star)$, i.e.,
\[x^k \to x_{\lambda,\delta}^\star, \quad \text{with } \nabla f^{\lambda,\delta}(x_{\lambda,\delta}^\star)=0,\]
and we have
\[\|\nabla f^{\lambda,\delta}(x^k)\|\to 0, \quad \text{with } \sum_k \|\nabla f^{\lambda,\delta}(x^k)\|< +\infty.\]
\end{theorem}
\begin{proof}
Since $f$ is bounded from below and $e^{-f/\delta}\in L^1(\mathbb R^d)$, the function $e^{-f/\delta}$ is positive and belongs to $L^1\cap L^\infty$. Hence its convolution with the gaussian heat-kernel
\[
u^\delta(\lambda,x):=\frac{1}{(2\pi \lambda \delta)^{d/2}}\int_{\mathbb R^d}
\exp\!\left(-\frac{f(y)}{\delta}\right)
\exp\!\left(-\frac{\|x-y\|^2}{2\lambda\delta}\right)\,dy
\]
is positive \cite[Section 2.3.1]{evans}, real analytic \cite[Section 2.3.3]{evans} (or more generally \cite[Theorem IX.13]{ReedSimon1975}), and vanishes at infinity \cite[Proposition 8.8]{folland1999real}. Therefore, since $f^{\lambda,\delta}(x)=
-\delta\log u^\delta(\lambda,x)$, we have
\(f^{\lambda,\delta}(x)\to+\infty\) as \(\|x\|\to+\infty\), so
\(f^{\lambda,\delta}\) is coercive. Since \(f^{\lambda,\delta}(x^k)\) is
decreasing, the sequence \(\{x^k\}_k\) is bounded. Notice that the previous argument proves also that $f^{\lambda,\delta}$ admits at least a minimizer. Since \(f^{\lambda,\delta}\) is real analytic, it satisfies the
{\L}ojasiewicz gradient inequality around each critical point
\cite{Lojasiewicz1963}; equivalently, it satisfies the KL
property, see also \cite{BolteDaniilidisLewis2007}. Combining the sufficient descent estimate of Theorem~\ref{thm:descent_on_f} (i) with
\[
\|\nabla f^{\lambda,\delta}(x^k)\|
=
\frac{\|x^{k+1}-x^k\|}{\lambda}
\]
the continuity of $f^{\lambda, \delta}$ and the KL inequality, the standard KL convergence theorem \cite[Theorem 2.9]{attouch2013convergence} yields the finite-length property $\sum_k \|x^{k+1}-x^k\| < +\infty$. Thus \(\{x^k\}_k\) is Cauchy and converges to some
\(x_{\lambda,\delta}^\star\). Finally, since the map $\zprox_{\lambda, f}^\delta$ is continuous, and from
\(x^{k+1}-x^k\to 0\) we obtain
$\zprox_{\lambda, f}^\delta(x_{\lambda,\delta}^\star)=x_{\lambda,\delta}^\star$ and equivalently $
\nabla f^{\lambda,\delta}(x_{\lambda,\delta}^\star)=0$.
\end{proof}

\subsection{Proof of Theorem~\ref{thm:rates}}\label{app:proof_of_rates}

In the following we use the Poincaré inequality for the measure $\mu_x$ defined by
\begin{equation}\label{eq:def_mux}
    d\mu_x:=\frac{1}{Z_{\lambda,\delta}(x)}\exp\!\left(-\frac{f(y)}{\delta}\right)\,
\exp\!\left(-\frac{\|y-x\|^2}{2\lambda\delta}\right)\,dy,
\end{equation}
which we define here (for a comprehensive treatment, see \cite[Section 4.2]{bakry2014analysis}).

\begin{definition}
    A probability measure $\mu$ satisfies a Poincaré inequality if there exists a constant $C_P > 0$ such that for every smooth, suitably integrable function $f$,
    \[\operatorname{Var}_{\mu}(f) = \int \left|f(y)-\int f \, d\mu\right|^2\, d\mu(y)\leq C_P \int \|\nabla f(y)\|^2 \, d\mu(y).\]
\end{definition}

Here we state a lemma which we will use several times later on.

\begin{lemma}[Covariance bound via Poincaré]\label{lem:covariance_poincare}
If $\mu_x$ in \eqref{eq:def_mux} satisfies a Poincaré inequality with constant $C_P(\delta, x) < \infty$, then its covariance matrix 
\[\Sigma_{\lambda,\delta}(x) := \int_{\mathbb{R}^d}
\big(y-\zprox_{\lambda,f}^\delta(x)\big)
\big(y-\zprox_{\lambda,f}^\delta(x)\big)^\top
\,d\mu_x(y)\]
is bounded as
\begin{equation}
\Sigma_{\lambda,\delta}(x) \preceq C_P(\delta, x) I.
\end{equation}
Consequently, the trace of the covariance matrix satisfies
\[\int_{\mathbb{R}^d}
\|y-\zprox_{\lambda,f}^\delta(x)\|^2
\,d\mu_x(y)\leq d \, C_P(\delta, x).\]
\end{lemma}

\begin{proof}
First, recall that $\zprox_{\lambda,f}^\delta(x)=\E_{\mu_x}[y]$. By the definition of the Poincaré inequality, the variance of any sufficiently smooth function $g$ under $\mu_x$ is bounded by $C_P(\delta, x) \int \|\nabla g(y)\|^2 d\mu_x(y)$. Applying this to the linear projection $g(y) = u^\top y$ for any unit vector $u \in \mathbb{R}^d$ ($\|u\|=1$), we get $\nabla g(y) = u$. Thus, the variance in direction $u$ is $u^\top \Sigma_{\lambda,\delta}(x) u \leq C_P(\delta, x) \|u\|^2 = C_P(\delta, x)$, which proves the matrix inequality. Summing the variances over the $d$ standard basis vectors yields the trace bound.
\end{proof}

\begin{definition}[$L$-smooth function]
    A function $f:\R^d\to \R$ is said to be $L$-smooth if it is $C^1$ and $\nabla f$ is $L$-Lipschitz, i.e. it satisfies
    \[\forall x, y \in \R^d, \quad \|\nabla f(x) - \nabla f(y)\|\leq L||x-y||.\]
\end{definition}

\begin{theorem}\label{thm:quasi_min_general}
Let $f:\mathbb{R}^d\to \mathbb{R}$ be $L$-smooth. Let $\lambda, \delta > 0$, and assume that for any given $x \in \mathbb{R}^d$, the localized Gibbs measure $\mu_x$ defined by the density
$$ \mu_x(y) \propto \exp\left(-\frac{f(y) + \frac{1}{2\lambda}\|x-y\|^2}{\delta}\right) $$
satisfies a Poincaré inequality with constant $C_P(\delta, x) < \infty$. Then for all $x\in \R^d$ it holds
$$ \|\nabla f^{\lambda,\delta}(x)-\nabla f(x)\| \leq L\lambda \|\nabla f^{\lambda,\delta}(x)\| + L\sqrt{d\,C_P(\delta, x)} . $$
In particular, the sequence $\{x^k\}_k$ generated by Algorithm \ref{alg:zoppa} satisfies
$$ \|\nabla f(x^k)\| \leq (1+L\lambda)\|\nabla f^{\lambda, \delta}(x^k)\|+L\sqrt{d \, C_P(\delta, x^k)}.$$
\end{theorem}

\begin{proof}
Since $\nabla f$ is $L$-Lipschitz, we have $\|\nabla f(y)\|e^{-f(y)/\delta}\leq C e^{-\min f/\delta}(1+\|y\|)$, for some $C>0$. Thus, we can differentiate under the integral, and the gradient of the soft envelope can be expressed as the expected gradient of the original function under this measure:
$$ \nabla f^{\lambda,\delta}(x) = \int_{\mathbb{R}^d} \nabla f(y) \, d\mu_x(y) . $$
Therefore, we can write the difference between the gradients as
$$ \nabla f^{\lambda,\delta}(x) - \nabla f(x) = \int_{\mathbb{R}^d} \big(\nabla f(y) - \nabla f(x)\big) \, d\mu_x(y) . $$
Taking the norm and applying Jensen's inequality to move the norm inside the integral, we obtain
$$ \|\nabla f^{\lambda,\delta}(x) - \nabla f(x)\| \leq \int_{\mathbb{R}^d} \|\nabla f(y) - \nabla f(x)\| \, d\mu_x(y) \leq L \int_{\mathbb{R}^d} \|y-x\| \, d\mu_x(y) , $$
by the assumption that $f$ is $L$-smooth. Applying the Cauchy-Schwarz inequality gives
$$ \|\nabla f^{\lambda,\delta}(x) - \nabla f(x)\| \leq L \left( \int_{\mathbb{R}^d} \|y-x\|^2 \, d\mu_x(y) \right)^{1/2} . $$
To evaluate this integral, we decompose the squared distance using the mean of the measure. Recall that $\zprox_{\lambda, f}^\delta(x) = x - \lambda \nabla f^{\lambda,\delta}(x) = \int y \, d\mu_x(y)$. Thus
$$ \int_{\mathbb{R}^d} \|y-x\|^2 \, d\mu_x(y) = \int_{\mathbb{R}^d} \|y - \zprox_{\lambda, f}^\delta(x)\|^2 \, d\mu_x(y) + \lambda^2\|\nabla f^{\lambda,\delta}(x)\|^2 , $$
since the cross term vanishes. The first term is the trace of the covariance matrix of $\mu_x$. By the assumption that $\mu_x$ satisfies a Poincaré inequality with constant $C_P(\delta, x)$, the variance of any coordinate projection under $\mu_x$ is bounded by $C_P(\delta, x)$. Summing over the $d$ dimensions, the trace of the covariance matrix is strictly bounded by $d C_P(\delta, x)$. Substituting this term back into our bound yields
$$ \|\nabla f^{\lambda,\delta}(x) - \nabla f(x)\| \leq L \left( d C_P(\delta, x) + \lambda^2 \|\nabla f^{\lambda,\delta}(x)\|^2 \right)^{1/2} . $$
Using the subadditivity of the square root, we obtain the global bound
$$ \|\nabla f^{\lambda,\delta}(x) - \nabla f(x)\| \leq L\lambda \|\nabla f^{\lambda,\delta}(x)\| + L\sqrt{d C_P(\delta, x)} . $$
The second part follows immediately.
\end{proof}

\begin{proposition}[Some possible bounds for the Poincaré constant]\label{prop:bounds_on_poincare}
Let $f:\R^d\to \R$ satisfy Assumption~\ref{ass:on_f}. Then, 
    \begin{enumerate}
        \item[(i)] if $f$ is $L$-smooth and $\lambda < 1/L$, for any $x\in \R^d$ it holds
        \[C_P(\delta, x) \leq \frac{\lambda\delta}{1-\lambda L}.\]
        \item[(ii)] if $f$ is $G$-Lipschitz, for any $\lambda >0$ and $x\in \R^d$ it holds
\[
C_P(\delta,x)\le
\min\left\{
2\lambda\delta \exp\!\Bigl(4\sqrt{2d/\pi}\,G^2\lambda^2\Bigr),\;
\frac12\Bigl(2G\lambda+\sqrt{8\lambda\delta}\Bigr)^2
\exp\!\Bigl(\frac{G^2\lambda}{2\delta}\Bigr)
\right\},
\]
\item[(iii)] if $f(y)=\frac{\kappa}{2}\|y\|^2+V(y)$ with $\operatorname{osc}(V)= \sup V - \inf V <+\infty$, for any $\lambda>0$ and $x\in \R^d$ it holds
\[
C_P(\delta,x)\leq e^{\osc(V)/\delta}
\frac{\lambda\delta}{1+\kappa\lambda}.
\]
    \end{enumerate}
\end{proposition}

\begin{proof}
    (i) Because $f$ is $L$-smooth, it is $L$-weakly convex. Thus, the function $\varphi_x(y) := f(y) + \frac{1}{2\lambda}\|x-y\|^2$ is strongly convex with parameter $\rho = \frac{1}{\lambda} - L = \frac{1-\lambda L}{\lambda} > 0$. By the Bakry-Émery criterion \cite[Proposition 4.8.1 and Corollary 4.8.2]{bakry2014analysis} (see also the simpler \cite{chewi2024log}), the measure $\mu_x$ satisfies a Poincaré inequality with constant $C_P(\delta, x) \leq \frac{\delta}{\rho} = \frac{\lambda\delta}{1-\lambda L}$.\\
    (ii) The term $\frac{1}{2\lambda\delta}\|x-y\|^2$ is $(\lambda\delta)^{-1}$-strongly convex.
Since $f$ is $G$-Lipschitz, $f/\delta$ is
$(G/\delta)$-Lipschitz and
applying \cite[Theorem 1.3(1) and (2)]{cattiaux2022perturbed} to the Gaussian measure with density $\propto e^{-\frac{1}{2\lambda\delta}\|x-y\|^2}$ and perturbed by the Lipschitz function $f/\delta$ yields the bound
\[
C_P(\delta,x)\le
\min\left\{
2\lambda\delta \exp\!\Bigl(4\sqrt{2d/\pi}\,G^2\lambda^2\Bigr),\;
\frac12\Bigl(2G\lambda+\sqrt{8\lambda\delta}\Bigr)^2
\exp\!\Bigl(\frac{G^2\lambda}{2\delta}\Bigr)
\right\},
\]
for all $x$.\\
(iii) We can decompose the measure $\mu_x$ into a strongly log-concave base measure perturbed by a bounded function. Let the base measure be $\gamma_x(dy) \propto \exp\left(-\frac{1}{\delta}\big[\frac{\kappa}{2}\|y\|^2 + \frac{1}{2\lambda}\|y-x\|^2\big]\right) dy$. The effective potential of $\gamma_x$ is strongly convex with parameter $\rho = \frac{1}{\delta}(\kappa + \frac{1}{\lambda})$. By the Bakry-Émery criterion \cite[Proposition 4.8.1 and Corollary 4.8.2]{bakry2014analysis}, $\gamma_x$ satisfies a Poincaré inequality with constant $C_P(\gamma_x) \le 1/\rho = \frac{\lambda \delta}{1+\kappa\lambda}$. The target measure $\mu_x$ is obtained by perturbing $\gamma_x$ with the bounded density $\exp(-V(y)/\delta)$. By the Holley-Stroock perturbation theorem \cite{holley1987} (see also \cite[Theorem 1.1]{cattiaux2022perturbed}), $\mu_x$ satisfies a Poincaré inequality with a constant bounded by the oscillation of the perturbation
\[
C_P(\delta, x) \le e^{\osc(V) / \delta} C_P(\gamma_x) \le \frac{\lambda \delta}{1+\kappa\lambda} e^{\osc(V) / \delta}.
\]
\end{proof}

\begin{corollary}[Theorem~\ref{thm:rates}]\label{cor:quasi_min_small_lambda}
    Let $f:\R^d\to \R$ satisfy Assumption~\ref{ass:on_f} and $L$-smooth. Then, 
    \begin{enumerate}
        \item[(i)] if $\lambda < 1/L$, it holds
        \[\|\nabla f(x^k)\| \leq (1+\lambda L)\| \nabla f^{\lambda,\delta}(x^k)\| + L \sqrt{\frac{d \lambda \delta}{1-\lambda L}}.\]
        \item[(ii)] if $f$ is also $G$-Lipschitz, for any $\lambda >0$ it holds
\[\begin{aligned}\|\nabla f(x^k)\| \leq & (1+\lambda L)\| \nabla f^{\lambda,\delta}(x^k)\| \\
&+ L\,
\min\left\{
\sqrt{2d \lambda\delta}\,
\exp\!\Bigl(2\sqrt{2d/\pi}\,G^2\lambda^2\Bigr),\;
\left(\sqrt{2 d}G\lambda+2\sqrt{d\lambda\delta}\right)\,
\exp\!\Bigl(\frac{G^2\lambda}{4\delta}\Bigr)
\right\}.
\end{aligned}\]
\item[(iii)] if $f(y)=\frac{\kappa}{2}\|y\|^2+V(y)$ with $\operatorname{osc}(V)= \sup V - \inf V <+\infty$, for any $\lambda>0$ it holds
\[
\|\nabla f(x^k)\|
\le
(1+\lambda L)\|\nabla f^{\lambda,\delta}(x^k)\|
+
L\sqrt{
d\,e^{\operatorname{osc}(V)/\delta}
\frac{\lambda\delta}{1+\kappa\lambda}
}.
\]
    \end{enumerate}
\end{corollary}
\begin{proof}
    Substituting the bounds of the previous theorem into Theorem~\ref{thm:quasi_min_general} directly provides the explicit bounds.
\end{proof}

\begin{remark}
    Notice that the result above can also be converted to control on how good for $f$ is the limit point $x_{\lambda,\delta}^\star$ by just substituting $x^k$ with $x_{\lambda,\delta}^\star$ and $\|\nabla f^{\lambda,\delta}(x^k)\|$ with $0$.
\end{remark}

\subsection{Convexification in $\lambda$: proof of Theorem~\ref{thm:convexification}}\label{app:proof_of_convexification}

The bounds in Theorem~\ref{thm:rates} highlight the role of the parameter $\lambda$. Small values of $\lambda$ yield sharper guarantees on $\|\nabla f(x^k)\|$, as the approximation error vanishes with $\lambda$. We show here that, conversely, larger values of $\lambda$ mitigate nonconvexity and simplifies the geometry of the objective.

\begin{theorem}[Hessian identity]\label{thm:hessian}
Let $f$ satisfy Assumption~\ref{ass:on_f}. It holds
\begin{equation}
\nabla^2 f^{\lambda,\delta}(x)=
\frac{1}{\lambda}I
-
\frac{1}{\lambda^2\delta}\,\Sigma_{\lambda,\delta}(x),
\label{eq:hessian-soft-envelope_text}
\end{equation}
where $\Sigma_{\lambda,\delta}(x)
:=
\int_{\mathbb{R}^d}
\big(y-\zprox_{\lambda,f}^\delta(x)\big)
\big(y-\zprox_{\lambda,f}^\delta(x)\big)^\top
\,d\mu_x(y)
\;\succeq\;0$. In particular, if $\Sigma_{\lambda,\delta}(x) \leq \lambda \delta I$, then $f^{\lambda,\delta}$ is convex at $x$.
\end{theorem}

\begin{proof}
The derivation under the integral sign in the following are justified, see Remark~\ref{rem:differentiation}. Let $T_{\lambda,\delta}:=\zprox_{\lambda,f}^\delta$ and
\begin{equation}\label{eq:qx_def_app}
    q_x(y) := \frac{1}{Z_{\lambda,\delta}(x)}\exp\left(-\frac{f(y)}{\delta}\right)\exp\left(-\frac{\|x-y\|^2}{2\lambda\delta}\right).
\end{equation} we have 
\[\nabla_x \log q_x(y)=-\nabla_x\left(\frac{\|x-y\|^2}{2\lambda\delta}\right)-\nabla_x\log Z_{\lambda,\delta}(x)
=
-\frac{x-y}{\lambda\delta}-\nabla_x\log Z_{\lambda,\delta}(x)\] 
and thus, since $-\nabla_x \log Z_{\lambda,\delta}(x)=\frac{1}{\delta}\nabla_x f^{\lambda,\delta}(x)$ and by \eqref{eq:zprox_softenvelope_app}
\[
\nabla_x \log q_x(y)
=
-\frac{x-y}{\lambda\delta}
+\frac{x-T_{\lambda,\delta}(x)}{\lambda\delta}
=
\frac{y-T_{\lambda,\delta}(x)}{\lambda\delta}.
\]
Since $\nabla_x q_x(y)=q_x(y)\nabla_x\log q_x(y)$, we have
\begin{align*}
D_xT_{\lambda,\delta}(x)
&=D_x\!\left(\int_{\mathbb{R}^d} y\,q_x(y)\,dy\right)
=\int_{\mathbb{R}^d} y\,(\nabla_x q_x(y))^\top\,dy
\\
&=\int_{\mathbb{R}^d} y\,q_x(y)\,\nabla_x\log q_x(y)^\top\,dy
=\frac{1}{\lambda\delta}\int_{\mathbb{R}^d} y\,\big(y-T_{\lambda,\delta}(x)\big)^\top\,q_x(y)\,dy.
\end{align*}
Decompose $y=\big(y-T_{\lambda,\delta}(x)\big)+T_{\lambda,\delta}(x)$ and use the fact that $\int_{\mathbb{R}^d} \big(y-T_{\lambda,\delta}(x)\big)\,q_x(y)\,dy=0$ to get
\[
\int_{\mathbb{R}^d} y\,\big(y-T_{\lambda,\delta}(x)\big)^\top q_x(y)\,dy
=
\int_{\mathbb{R}^d} \big(y-T_{\lambda,\delta}(x)\big)\big(y-T_{\lambda,\delta}(x)\big)^\top q_x(y)\,dy
=
\Sigma_{\lambda,\delta}(x).
\]
Hence $DT_{\lambda,\delta}(x)=\frac{1}{\lambda\delta}\Sigma_{\lambda,\delta}(x)$. Recalling that
\[
\nabla f^{\lambda,\delta}(x)
=
-\delta\nabla\log Z_{\lambda,\delta}(x)
=
\frac{1}{\lambda}\big(x-T_{\lambda,\delta}(x)\big),
\] 
differentiating and using the expression for $DT_{\lambda,\delta}(x)$ yields
\[
\nabla^2 f^{\lambda,\delta}(x)
=
\frac{1}{\lambda}\big(I-DT_{\lambda,\delta}(x)\big)
=
\frac{1}{\lambda}I-\frac{1}{\lambda^2\delta}\Sigma_{\lambda,\delta}(x).
\]
\end{proof}

Identity \eqref{eq:hessian-soft-envelope_text} separates two competing effects: a positive curvature term
$\lambda^{-1}I$ and a negative correction proportional to the matrix $\Sigma_{\lambda,\delta}(x)$.
Increasing $\lambda$ suppresses the negative term at rate $\lambda^{-2}$ while the positive term decays only as
$\lambda^{-1}$; thus, larger $\lambda$ could progressively remove negative curvature and simplifies the landscape of
$f^{\lambda,\delta}$. A sufficient condition for convexity is $\Sigma_{\lambda,\delta}(x) \preceq \lambda\delta I$ for all $x$.
More generally, if there exists $M<\infty$ such that $\sup_x \lambda_{\max}(\Sigma_{\lambda,\delta}(x))\le M$, then for any
$\lambda \geq M/\delta$ the envelope is convex
\[
\nabla^2 f^{\lambda,\delta}(x)
\succeq
\left(\frac{1}{\lambda}-\frac{M}{\lambda^2\delta}\right)I
=
\frac{\lambda\delta-M}{\lambda^2\delta}I\geq 0,
\qquad \forall x.
\]
In the following we find some cases where these bounds and conditions on $\lambda$ are explicit.

\begin{corollary}[Local convexification via dissipativity]\label{cor:convexification_diss}    
    Let $f:\mathbb{R}^d \to \mathbb{R}$ satisfy Assumption~\ref{ass:on_f} and the dissipativity condition, i.e., there exists $m > 0$ and $b \geq 0$ such that $\langle \nabla f(y), y \rangle \geq m\|y\|^2 - b$ for all $y \in \R^d$. Then, for any radius $R > 0$, $f^{\lambda,\delta}$ is convex on $B(0, R)$ provided
    \begin{equation}
        \lambda \geq \frac{\delta d + b + \sqrt{(\delta d + b)^2 + 2\delta m R^2}}{2\delta m}.
    \end{equation}
\end{corollary}

\begin{proof} Here we do not go through the Poincarè inequality since it would result in worse bounds. Instead, we estimate directly the variance. Let $\varphi_x(y) = f(y) + \frac{1}{2\lambda}\|y-x\|^2$, so that $q_x(y) \propto \exp(-\varphi_x(y)/\delta)$. Taking the inner product of $\nabla \varphi_x(y)$ with $y$ and applying the dissipativity condition alongside Young's inequality $\frac{1}{\lambda}\langle x, y \rangle \le \frac{1}{2\lambda}\|x\|^2 + \frac{1}{2\lambda}\|y\|^2$ yields
\[
\langle \nabla \varphi_x(y), y \rangle \ge m\|y\|^2 - b + \frac{1}{2\lambda}\|y\|^2 - \frac{1}{2\lambda}\|x\|^2.
\]
For any $x \in B(0, R)$, we drop the positive $\frac{1}{2\lambda}\|y\|^2$ term and bound $\|x\|^2 \le R^2$ to obtain the uniform lower bound:
\begin{equation}\label{eq:dissipativity_bound_proof}
\langle \nabla \varphi_x(y), y \rangle \ge m\|y\|^2 - \left( b + \frac{R^2}{2\lambda} \right).
\end{equation}
Because $\varphi_x$ grows quadratically, $q_x$ decays exponentially, permitting integration by parts
\[
\int_{\mathbb{R}^d} \langle \nabla \varphi_x(y), y \rangle e^{-\varphi_x(y)/\delta} dy = -\delta \int_{\mathbb{R}^d} y \cdot \nabla \left( e^{-\varphi_x(y)/\delta} \right) dy = \delta d \int_{\mathbb{R}^d} e^{-\varphi_x(y)/\delta} dy.
\]
Dividing by the normalization constant gives the exact identity $\mathbb{E}_{\mu_x}[\langle \nabla \varphi_x(y), y \rangle] = \delta d$. Taking the expectation of \eqref{eq:dissipativity_bound_proof} yields a uniform bound on the uncentered second moment
\[
\delta d \ge m \mathbb{E}_{\mu_x}[\|y\|^2] - \left( b + \frac{R^2}{2\lambda} \right) \implies u^\top \Sigma_{\lambda,\delta}(x) u \le \mathbb{E}_{\mu_x}[\|y\|^2] \le \frac{1}{m}\left( \delta d + b + \frac{R^2}{2\lambda} \right).
\]
Substituting this uniform covariance bound into \eqref{eq:hessian-soft-envelope_text} implies that $f^{\lambda,\delta}$ is convex on $B(0,R)$ if
\[
\frac{1}{\lambda} - \frac{1}{\lambda^2\delta m}\left( \delta d + b + \frac{R^2}{2\lambda} \right) \ge 0.
\]
Multiplying by $\delta m \lambda^3$, this simplifies to the quadratic inequality $\delta m \lambda^2 - (\delta d + b)\lambda - \frac{R^2}{2} \ge 0$. The roots of this parabola are $\frac{(\delta d + b) \pm \sqrt{(\delta d + b)^2 + 2\delta m R^2}}{2\delta m}$. Since the leading coefficient $\delta m$ is positive, the inequality holds for all $\lambda$ greater than or equal to the positive root.
\end{proof}

\begin{remark}
    The dissipativity condition is a standard assumption in the literature of non-convex optimization and sampling, particularly within the analysis of Langevin dynamics (see, e.g., \cite{Raginsky2017, XCZG2018}). Geometrically, it ensures that the gradient is approximately aligned with the position vector outside of bounded sets. This ensures that the tails of the associated localized Gibbs measure decay rapidly enough to admit a uniformly bounded covariance matrix.
\end{remark}

\begin{proposition}[Convexification via Poincaré]\label{prop:convexification_via_poincare}
Assume that $\mu_x$ as in \eqref{eq:def_mux} admits a uniform Poincaré constant  $C_P(\delta) < \infty$ such that $C_P(\delta, x) \le C_P(\delta)$ for all $x \in \mathbb{R}^d$. If the stepsize $\lambda$ satisfies
$$ \lambda \geq \frac{C_P(\delta)}{\delta}, $$
then $f^{\lambda,\delta}$ is convex.
\end{proposition}
\begin{proof}
By Lemma~\ref{lem:covariance_poincare} and the uniform bound, we have $\Sigma_{\lambda,\delta}(x) \preceq C_P(\delta) I$. Substituting this into the Hessian identity of Theorem~\ref{thm:hessian}
$$ \nabla^2 f^{\lambda,\delta}(x) = \frac{1}{\lambda}I - \frac{1}{\lambda^2\delta}\Sigma_{\lambda,\delta}(x),$$
we obtain
$$ \nabla^2 f^{\lambda,\delta}(x) \succeq \left( \frac{1}{\lambda} - \frac{C_P(\delta)}{\lambda^2\delta} \right) I. $$
For $f^{\lambda,\delta}$ to be convex, we require the right-hand side to be non-negative, which holds rearranged as $\lambda \ge C_P(\delta) / \delta$.
\end{proof}

\begin{corollary}[Convexification via bounded oscillations]\label{cor:convexification_osc}
Let $f$ of the form $f(y) = \frac{\kappa}{2}\|y\|^2 + V(y)$ for some $\kappa > 0$, where $V$ is continuous bounded with global oscillation $\osc (V) := \sup V - \inf V < \infty$ (in particular $f$ satisfies also Assumption~\ref{ass:on_f}). Then, $f^{\lambda,\delta}$ is globally convex provided
    \begin{equation}\label{eq:lambda_big_osc}
        \lambda \geq \frac{e^{\osc(V)/\delta}-1}{\kappa}
    \end{equation}
    \end{corollary}
\begin{proof} 
By Proposition~\ref{prop:bounds_on_poincare} it holds
\[
C_P(\delta,x)\leq e^{\osc(V)/\delta}
\frac{\lambda\delta}{1+\kappa\lambda}.
\]
and therefore by Proposition~\ref{prop:convexification_via_poincare} we have the condition $\lambda \geq \frac{\lambda}{1+\kappa \lambda}e^{\osc(V)/\delta}$ which translates to \eqref{eq:lambda_big_osc}.
\end{proof}

\begin{remark}[Comparison with Moreau envelope]
    Corollary~\ref{cor:convexification_osc} highlights a fundamental structural difference between the soft Moreau envelope ($\delta > 0$) and the classical Moreau envelope ($\delta \to 0$). For the classical Moreau envelope to be convex, the original function $f$ must typically be $\rho$-weakly convex (i.e., its negative curvature is globally bounded), and one must choose a sufficiently \emph{small} stepsize $\lambda < 1/\rho$. If $f$ possesses bounded oscillations but unbounded negative curvature (e.g., $V$ is continuous but nowhere differentiable), the classical envelope will generally fail to be convex for any choice of $\lambda$. In contrast, the soft envelope leverages the global integration to average out arbitrary bounded non-convexities. This entropic smoothing allows $f^{\lambda, \delta}$ to become globally convex for sufficiently \emph{large} $\lambda$, bypassing the need for weak convexity. The fact that the threshold in \eqref{eq:lambda_big_osc} diverges to $+\infty$ as $\delta \to 0$ perfectly captures the collapse of this phenomenon: in the deterministic zero-temperature limit, global macroscopic convexification via large stepsizes is impossible.
\end{remark}

In Figure~\ref{fig:convexification_app} we show convexification of various functions when $\lambda$ is larger, while noticing that this becomes harder when $\delta$ is small, further hinting to the fact that the $\delta\to 0$ regime considered in the literature is detrimental. Looking at the shapes and local minima for the various $f^{\lambda,\delta}$ these experiments hint also to the fact that the continuation strategy we propose in Section \ref{sec:exp_lambdarole} could be effective.

\begin{figure}[htbp]
    \centering
    \includegraphics[width=\linewidth]{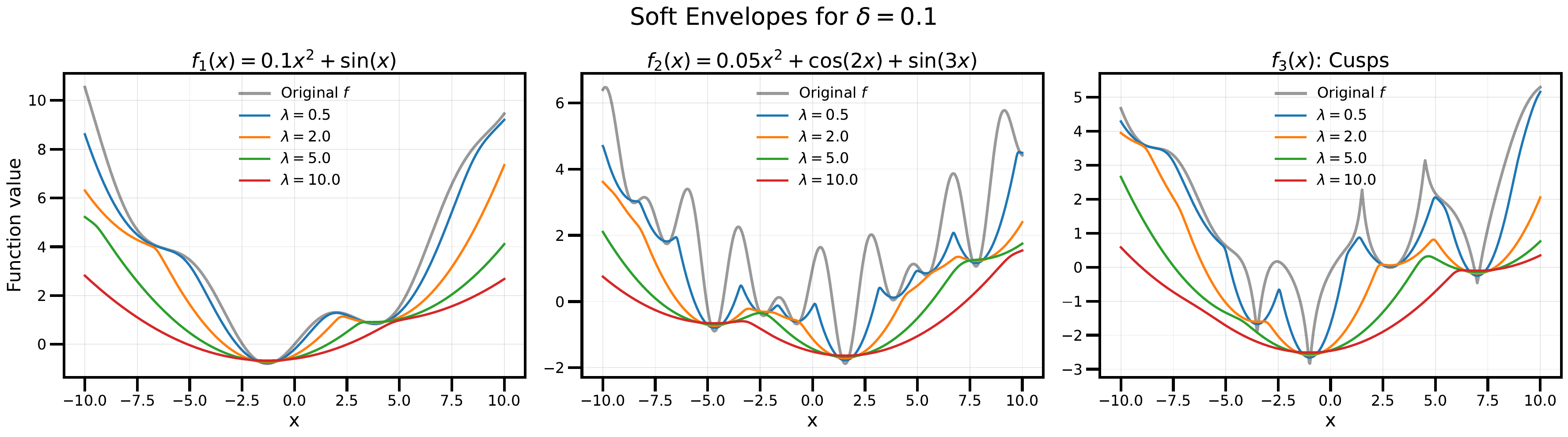} 

    \includegraphics[width=\linewidth]{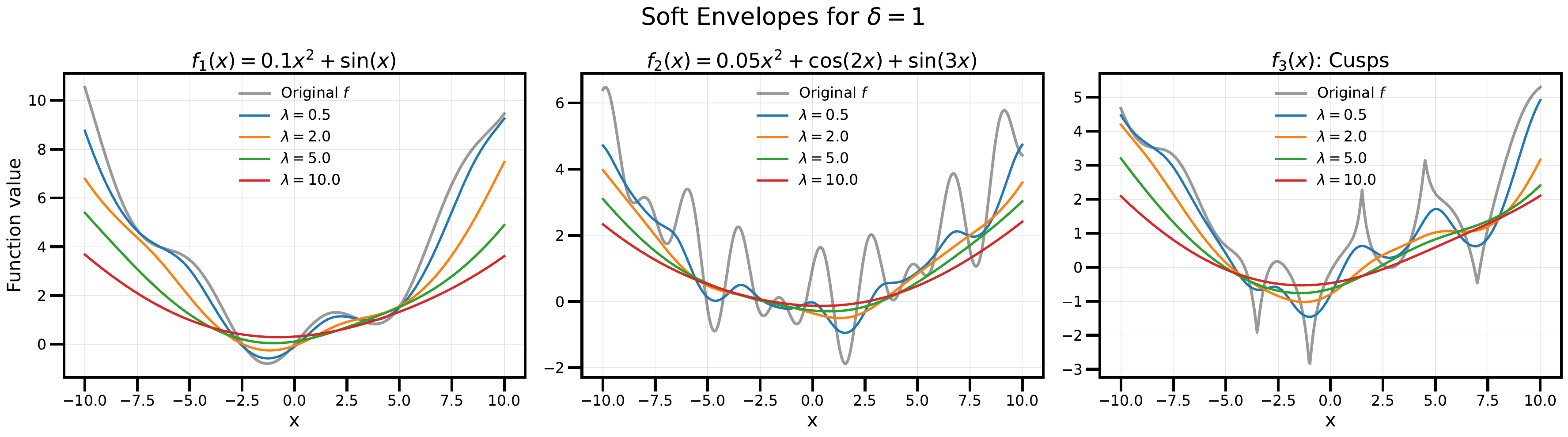} 

    \includegraphics[width=\linewidth]{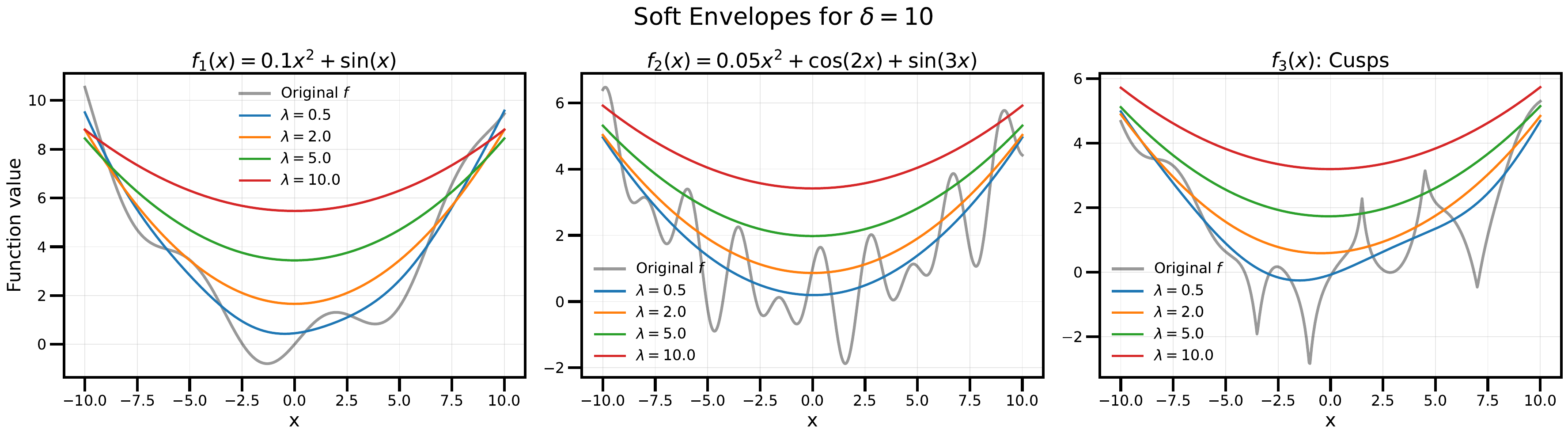} 

    \includegraphics[width=\linewidth]{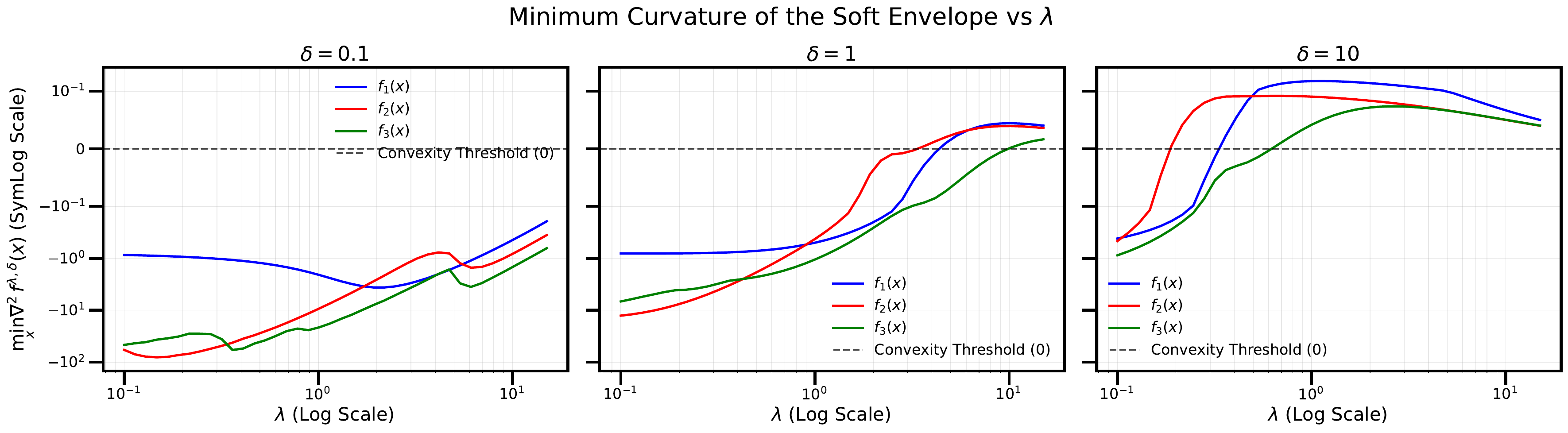}
    
    \caption{Example of convexification effect in $\lambda$. In this picture it is also shown the fact that going towards minimizers for $f^{\lambda,\delta}$ moves the iterates in low energy regions also for the original $f$. The convexification effect is quite quick in practice for $\delta$ not too small, while for smaller $\delta$ it requires large $\lambda$.}\label{fig:convexification_app}
\end{figure}

\section{The convex case}

\subsection{Firm nonexpansivity of the zeroth order prox: proof of Theorem~\ref{thm:zoppa_cvx}}\label{app:proof_of_zoppa_cvx}

\begin{theorem}\label{thm:fne}
Let $f:\mathbb{R}^d\to\mathbb{R}$ satisfy Assumption~\ref{ass:on_f} such that $f^{\lambda,\delta}$ is convex. Then there exists a  convex function $H_{\lambda,\delta}:\mathbb{R}^d\to\mathbb{R}$ such that
\[
\zprox_{\lambda, f}^{\delta}=\prox_{\lambda \delta H_{\lambda,\delta}},
\]
and $\zprox_{\lambda, f}^{\delta}$ is firmly nonexpansive. This happens in particular if $f$ is convex or $f$, $\lambda$, $\delta$ satisfy the conditions in Proposition~\ref{prop:convexification_via_poincare} or Corollary~\ref{cor:convexification_osc}.
\end{theorem}

\begin{proof} This is a direct application of Theorem~\ref{thm:zoppa_convergence} and \cite[Lemma 2]{Gribonval2013}, where it is stated that $f^{\lambda,\delta}$ is convex if and only if the correspondent $H_{\lambda,\delta}$ introduced in Theorem~\ref{thm:zoppa_convergence} is convex. The result can also be deduced from the spectral bounds $0\preceq \nabla^2 f^{\lambda,\delta}(x)\preceq \frac1\lambda I$, which ensure that the function $\psi(x) := \frac{1}{2}\|x\|^2 - \lambda f^{\lambda,\delta}$ is convex and has a gradient ($\nabla \psi = \zprox_{\lambda,f}^\delta$) that is $1$-Lipschitz. By the fundamental characterization of proximal mappings (see \cite[Proposition 1]{gribonval2020characterization, lauga2026} but also \cite[Corollary 10.c]{moreau1965}), a map is the proximal operator of a proper lower semicontinuous convex function if and only if it is the gradient of a convex lower semicontinuous function and is non-expansive. These conditions are satisfied here, and firm non-expansivity follows as an inherent property of proximal operators associated with convex penalties. Finally, in the special where $f$ is convex, we can use Proposition~\ref{prop:convexification_via_poincare} or also \cite[Proposition 23]{lauga2026} which states that when $f$ is convex, $f^{\lambda,\delta}$ is convex (log-concavity is preserved by convolution \cite[Proposition 3.5]{SaumardWellner2014}).
\end{proof}

\begin{corollary}
  Let $f:\mathbb{R}^d\to\mathbb{R}$ satisfy Assumption~\ref{ass:on_f} such that $f^{\lambda,\delta}$ is convex and denote $X^\star_{\lambda,\delta}:=\argmin f^{\lambda,\delta}$. Then, the sequence $(x^k)_{k\in\mathbb{N}}$ generated by Algorithm \ref{alg:zoppa} converges to some $x^\star_{\lambda,\delta}\in X_{\lambda,\delta}^{\star}$, and for every $k\ge1$,
\begin{equation}\label{eq:GDrates_app}
    f^{\lambda,\delta}(x^k)-\min f^{\lambda,\delta}
\le
\frac{\dist(x^0,X_{\lambda,\delta}^{\star})^2}{2\lambda k}, \quad \|\nabla f^{\lambda,\delta}(x^k)\|
\le
\frac{2\,\dist(x^0,X_{\lambda,\delta}^{\star})}{\lambda(k+1)}.
\end{equation}
\end{corollary}
\begin{proof}
From the previous theorem, we established that $f^{\lambda,\delta}$ is convex and $1/\lambda$-smooth. Consequently, ZOPPA is equivalent to a gradient descent algorithm with a fixed stepsize $\lambda = 1/L$. The convergence to a minimizer and the functional rates are standard in optimization literature, see \cite{Dunn_GD} and \cite[Theorem 10.21]{BeckBook}. The explicit rate for $\|\nabla f^{\lambda,\delta}(x^k)\|$ is less common in the literature, but it can be deduced for example from \cite[Theorem 10.27]{BeckBook}.
\end{proof}

\subsection{Rates: proof of Theorem~\ref{thm:zoppa_cvx_rates}}\label{app:proof_of_zoppa_cvx_rates}

\begin{theorem}[Non-smooth case]
Let $f:\mathbb{R}^d\to \mathbb{R}$ be a convex function satisfying Assumption \ref{ass:on_f}. For any $x \in \mathbb{R}^d$, the gradient of the soft Moreau envelope satisfies:
$$\nabla f^{\lambda,\delta}(x) \in \partial_{d\delta} f\left(\zprox_{\lambda, f}^\delta(x)\right),$$
where $\partial_\varepsilon f$ denotes the $\varepsilon$-subdifferential of $f$. This implies
\[f(x^{k})-\min_y f(y)\leq \|\nabla f^{\lambda,\delta}(x^{k-1})\|\operatorname{dist}(x^{k},X^\star)+ d \delta.\]
\end{theorem}
\begin{proof} Let \(q_x\) be the density defined in \eqref{eq:qx_def_app}, and write \[ \hat y := \zprox_{\lambda,f}^{\delta}(x) = \int_{\mathbb{R}^d} w q_x(w)\,dw . \] Since \(f\) is finite and convex on \(\mathbb{R}^d\), it is locally Lipschitz and differentiable almost everywhere by Rademacher's theorem. Hence, for every \(y\in\mathbb{R}^d\) and for a.e. \(w\in\mathbb{R}^d\), \[ f(y)\ge f(w)+\langle \nabla f(w),y-w\rangle . \] Multiplying by \(q_x(w)\) and integrating gives \[ f(y) \ge \int_{\mathbb{R}^d} f(w)q_x(w)\,dw + \int_{\mathbb{R}^d} \langle \nabla f(w),y-w\rangle q_x(w)\,dw . \] Let  $\varphi_x(w) := f(w)+\frac{1}{2\lambda}\|w-x\|^2$, so that $\nabla f(w) = \nabla \varphi_x(w)-\frac{w-x}{\lambda}$ for a.e. \(w\). Therefore, \[ \begin{aligned} \int_{\mathbb{R}^d} \langle \nabla f(w),y-w\rangle q_x(w)\,dw &= \int_{\mathbb{R}^d} \langle \nabla \varphi_x(w),y-w\rangle q_x(w)\,dw \\ &\quad -\frac{1}{\lambda} \int_{\mathbb{R}^d} \langle w-x,y-w\rangle q_x(w)\,dw . \end{aligned} \] Since \(f\) is finite and convex on \(\mathbb{R}^d\), it is locally Lipschitz and differentiable a.e. Moreover, \(\varphi_x\) is convex and grows at least quadratically because of the term \(\|w-x\|^2/(2\lambda)\). Hence the following integrations by parts are justified in the weak sense, for instance by applying the identities to smooth approximations of \(f\) and passing to the limit. Since \(q_x(w)=Z_{\lambda,\delta}(x)^{-1}e^{-\varphi_x(w)/\delta}\), and since the
boundary terms vanish, integration by parts gives
\[\begin{aligned}
\int_{\mathbb{R}^d}
\left\langle \nabla \varphi_x(w), y-w \right\rangle
e^{-\varphi_x(w)/\delta}\,dw & =
-\delta
\int_{\mathbb{R}^d}
\left\langle
\nabla\left(e^{-\varphi_x(w)/\delta}\right),
y-w
\right\rangle\,dw \\
& =
\delta
\int_{\mathbb{R}^d}
e^{-\varphi_x(w)/\delta}
\nabla\cdot(y-w)\,dw \\
& =
-d\delta
\int_{\mathbb{R}^d}
e^{-\varphi_x(w)/\delta}\,dw .
\end{aligned}
\]
Dividing by \(Z_{\lambda,\delta}(x)\), we obtain
\[
\int_{\mathbb{R}^d}
\left\langle \nabla \varphi_x(w), y-w \right\rangle
q_x(w)\,dw
=
-d\delta.
\]
For the remaining term, we can write
\[
\begin{aligned}
-\int_{\mathbb{R}^d}
\langle w-x,y-w\rangle q_x(w)\,dw &=
\int_{\mathbb{R}^d}
\langle w-x,w-y\rangle q_x(w)\,dw \\
& =
\int_{\mathbb{R}^d}
\left\langle
(w-\hat y)+(\hat y-x),
(w-\hat y)+(\hat y-y)
\right\rangle q_x(w)\,dw .
\end{aligned}
\]
Expanding the scalar product gives
\[
\begin{aligned}
&-\int_{\mathbb{R}^d}
\langle w-x,y-w\rangle q_x(w)\,dw \\
&=
\int_{\mathbb{R}^d}\|w-\hat y\|^2q_x(w)\,dw
+
\left\langle
\int_{\mathbb{R}^d}(w-\hat y)q_x(w)\,dw,
\hat y-y
\right\rangle \\
&\quad+
\left\langle
\hat y-x,
\int_{\mathbb{R}^d}(w-\hat y)q_x(w)\,dw
\right\rangle
+
\langle \hat y-x,\hat y-y\rangle .
\end{aligned}
\]
By the definition of \(\hat y\) we have \(\int_{\mathbb{R}^d}(w-\hat y)q_x(w)\,dw=0\) and the two middle terms vanish. Therefore,
\[
-\int_{\mathbb{R}^d}
\langle w-x,y-w\rangle q_x(w)\,dw
=
\int_{\mathbb{R}^d}\|w-\hat y\|^2q_x(w)\,dw
+
\langle \hat y-x,\hat y-y\rangle .
\]
Combining the previous identities yields \[ \begin{aligned} f(y) &\ge \int_{\mathbb{R}^d} f(w)q_x(w)\,dw + \frac{1}{\lambda} \int_{\mathbb{R}^d}\|w-\hat y\|^2 q_x(w)\,dw \\ &\quad + \left\langle \frac{\hat y-x}{\lambda}, \hat y-y \right\rangle - d\delta . \end{aligned} \] By Jensen's inequality, \[ \int_{\mathbb{R}^d} f(w)q_x(w)\,dw \ge f\left(\int_{\mathbb{R}^d} wq_x(w)\,dw\right) = f(\hat y), \] and the square term is non-negative. Therefore, \[ f(y) \ge f(\hat y) + \left\langle \frac{x-\hat y}{\lambda}, y-\hat y \right\rangle - d\delta . \] This is precisely the defining inequality \[ \frac{x-\hat y}{\lambda} \in \partial_{d\delta}f(\hat y). \] Since $ \frac{x-\hat y}{\lambda} = \nabla f^{\lambda,\delta}(x)$, we obtain \[ \nabla f^{\lambda,\delta}(x) \in \partial_{d\delta}f(\zprox_{\lambda,f}^{\delta}(x)). \] Finally, taking \(x=x^{k-1}\), so that \(x^k=\zprox_{\lambda,f}^{\delta}(x^{k-1})\), gives \[ \nabla f^{\lambda,\delta}(x^{k-1}) \in \partial_{d\delta}f(x^k). \] Thus, for any \(x^\star\in X^\star\), \[ f(x^\star) \ge f(x^k) + \langle \nabla f^{\lambda,\delta}(x^{k-1}), x^\star-x^k \rangle - d\delta . \] Rearranging and applying Cauchy--Schwarz yields \[ f(x^k)-\min f \le \|\nabla f^{\lambda,\delta}(x^{k-1})\| \operatorname{dist}(x^k,X^\star) + d\delta . \] \end{proof}

\begin{theorem}\label{thm:quasi_min_cvx}
Let $f$ convex, satisfy Assumption~\ref{ass:on_f} and be $L$-smooth. Then
    \[\|\nabla f(x^k)\| \leq (1+\lambda L) \|\nabla f^{\lambda, \delta}(x^k)\|+ L \sqrt{d \lambda \delta}.\]
    and additionally to \eqref{eq:rate_nonsmooth} it also holds
    \[f(x^k) - \min_y f(y) \leq \left((1+\lambda L) \|\nabla f^{\lambda, \delta}(x^k)\|+ L \sqrt{d \lambda \delta}\right)\dist(x^k,X^\star).\]
\end{theorem}

\begin{proof}
Because $f$ is convex, the function $\varphi_x(y) := f(y) + \frac{1}{2\lambda}\|x-y\|^2$ is strongly convex with parameter $\rho = \frac{1}{\lambda}> 0$. By the Bakry-Émery criterion \cite[Proposition 4.8.1 and Corollary 4.8.2]{bakry2014analysis}, the measure $\mu_x$ as in \eqref{eq:def_mux} satisfies a Poincaré inequality with constant $C_P(\delta, x) \leq \lambda\delta$. Substituting this bound into Theorem~\ref{thm:quasi_min_general} provides the first inequality. Finally, the convexity of the function $f$ provides
\[f(x^k) - \min_y f(y) \leq -\langle \nabla f(x^k) , x^\star - x^k\rangle,\]
and the thesis.
\end{proof}

\section{Proof of Theorem~\ref{thm:S-ZOPPA}}\label{app:proof_S-ZOPPA}

The main challenge of the proof of Theorem~\ref{thm:S-ZOPPA} is to bound uniformly the variance of the estimator \(\szprox^{\delta,N}_{\lambda, f}\). We do it in the following lemma.

\begin{lemma}\label{lem:boundedness}
    Let \(f:\mathbb{R}^d\to \R\) be continuous and lower bounded. Then, for any compact set \(K\subset\R^d\) and any \(\varepsilon>0\), there exists \(C_{K,\varepsilon}>0\) such that
    \[
    \forall x\in K,\quad
    \mathbb{E}\left[
    \left\|
    \szprox^{\delta,N}_{\lambda, f}(x)
    -
    \zprox^{\delta}_{\lambda, f}(x)
    \right\|^2
    \mathbf{1}_{\Omega_{x,\varepsilon}}
    \right]
    \le
    \frac{C_{K,\varepsilon}}{N},
    \]
    where
    \[
    \Omega_{x,\varepsilon}
    =
    \left\{
    \omega:
    \frac{1}{N}\sum_{i=1}^N\exp(-f(y_i)/\delta)\ge\varepsilon
    \right\},
    \]
    and \(\left(y_i\right)_{i=1}^N\) are the i.i.d. samples from
    \(\mathcal{N}(x,\lambda\delta I)\) generating
    \(\szprox^{\delta,N}_{\lambda, f}(x)\).
\end{lemma}

\begin{proof}
We start by introducing some notation. For i.i.d. samples
\(y_1,\ldots,y_N\) from \(\mathcal{N}(x,\lambda \delta I)\), we denote
\[
    \mathbf u_i
    =
    y_i\exp(-f(y_i)/\delta),
    \qquad
    v_i
    =
    \exp(-f(y_i)/\delta),
\]
\[
    \hat{\mathbf u}
    =
    \frac{1}{N}\sum_{i=1}^N\mathbf u_i,
    \qquad
    \hat v
    =
    \frac{1}{N}\sum_{i=1}^Nv_i,
\]
and
\[
\begin{gathered}
\mathbf u
=
(2\pi\lambda\delta)^{-\frac{d}{2}}
\int_{\mathbb{R}^d}
y\,
\exp\!\left(-\frac{f(y)}{\delta}\right)
\exp\!\left(-\frac{\|x-y\|^2}{2\lambda\delta}\right)\,dy,
\\
v
=
(2\pi\lambda\delta)^{-\frac{d}{2}}
\int_{\mathbb{R}^d}
\exp\!\left(-\frac{f(y)}{\delta}\right)
\exp\!\left(-\frac{\|x-y\|^2}{2\lambda\delta}\right)\,dy.
\end{gathered}
\]
Then
\[
    \mathbb{E}[\hat{\mathbf u}]=\mathbf u,
    \qquad
    \mathbb{E}[\hat v]=v,
\]
and, for a given set of samples,
\[
    \szprox^{\delta,N}_{\lambda, f}(x)
    =
    \frac{\hat{\mathbf u}}{\hat v},
    \qquad
    \zprox^{\delta}_{\lambda, f}(x)
    =
    \frac{\mathbf u}{v}.
\]

First note that \(\mathbf u\) and \(v\) are continuous functions of \(x\), and consequently bounded on any compact set. In addition, since \(f\) is finite-valued, \(v(x)>0\) for every \(x\), and hence \(v\) is bounded away from zero on \(K\). Furthermore, since \(f\) is bounded below, for every \(p\ge 1\),
\[
\sup_{x\in K}\mathbb{E}\|\mathbf u_1\|^p<+\infty,
\qquad
\sup_{x\in K}\mathbb{E}|v_1|^p<+\infty .
\]
Indeed, if \(f\ge \inf f\), then
\[
\begin{aligned}
&\int_{\mathbb{R}^d}(1+\|y\|^p)
\exp\left(-\frac{p f(y)}{\delta}\right)
\exp\!\left(-\frac{\|x-y\|^2}{2\lambda\delta}\right)dy \\
&\qquad\le
\exp\left(-\frac{p\inf f}{\delta}\right)
\int_{\mathbb{R}^d}(1+\|y\|^p)
\exp\!\left(-\frac{\|x-y\|^2}{2\lambda\delta}\right)dy,
\end{aligned}
\]
and the last quantity is uniformly bounded for \(x\in K\).

On the event \(\Omega_{x,\varepsilon}=\{\hat v\ge \varepsilon\}\), the algebraic expansion
\[
\frac{\hat{\mathbf u}}{\hat v}-\frac{\mathbf u}{v}
=
\frac{\hat{\mathbf u}-\mathbf u}{v}
-
\frac{\mathbf u(\hat v-v)}{v^2}
-
\frac{(\hat{\mathbf u}-\mathbf u)(\hat v-v)}{v\hat v}
+
\frac{\mathbf u(\hat v-v)^2}{v^2\hat v}
\]
gives the pointwise bound
\[
\left\|
\szprox_{\lambda,f}^{\delta,N}(x)
-
\zprox_{\lambda,f}^{\delta}(x)
\right\|
\le A_1+A_2+A_3+A_4,
\]
where
\[
A_1:=\frac{\|\hat{\mathbf u}-\mathbf u\|}{v},
\quad
A_2:=\frac{\|\mathbf u\|\,|\hat v-v|}{v^2}, \quad A_3:=
\frac{\|\hat{\mathbf u}-\mathbf u\|\,|\hat v-v|}{v\varepsilon},
\quad
A_4:=
\frac{\|\mathbf u\|\,|\hat v-v|^2}{v^2\varepsilon}.
\]
Therefore, using \((\sum_{j=1}^4 A_j)^2\le 4\sum_{j=1}^4 A_j^2\), we obtain
\[
\begin{aligned}
\mathbb{E}\left[
\left\|
\szprox_{\lambda,f}^{\delta,N}(x)
-
\zprox_{\lambda,f}^{\delta}(x)
\right\|^2
\mathbf 1_{\Omega_{x,\varepsilon}}
\right] &\le
\frac{4}{v^2}\mathbb{E}\|\hat{\mathbf u}-\mathbf u\|^2
+
\frac{4\|\mathbf u\|^2}{v^4}\mathbb{E}|\hat v-v|^2 \\
&\quad+
\frac{4}{v^2\varepsilon^2}
\mathbb{E}\left[
\|\hat{\mathbf u}-\mathbf u\|^2|\hat v-v|^2
\right]
+
\frac{4\|\mathbf u\|^2}{v^4\varepsilon^2}
\mathbb{E}|\hat v-v|^4 .
\end{aligned}
\]
Since \(\hat{\mathbf u}\) and \(\hat v\) are empirical averages of i.i.d. random variables,
\[
\mathbb{E}\|\hat{\mathbf u}-\mathbf u\|^2
=
\frac{1}{N}\mathbb{E}\|\mathbf u_1-\mathbf u\|^2,
\qquad
\mathbb{E}|\hat v-v|^2
=
\frac{1}{N}\mathbb{E}|v_1-v|^2.
\]
Moreover, the fourth moments of \(\mathbf u_1\) and \(v_1\) are finite and uniformly bounded for \(x\in K\). Since the variables \(\mathbf u_i-\mathbf u\) and \(v_i-v\) are centered and i.i.d., the elementary estimate
\[
\mathbb{E}\left\|\frac1N\sum_{i=1}^N \xi_i\right\|^4
\le
\frac{C}{N^2}\mathbb{E}\|\xi_1\|^4
\]
for centered i.i.d. variables gives
\[
\mathbb{E}\|\hat{\mathbf u}-\mathbf u\|^4
\le
\frac{C_K}{N^2},
\qquad
\mathbb{E}|\hat v-v|^4
\le
\frac{C_K}{N^2}.
\]
By Cauchy--Schwarz,
\[
\mathbb{E}\left[
\|\hat{\mathbf u}-\mathbf u\|^2|\hat v-v|^2
\right]
\le
\sqrt{
\mathbb{E}\|\hat{\mathbf u}-\mathbf u\|^4
\mathbb{E}|\hat v-v|^4
}
\le
\frac{C_K}{N^2}.
\]
Since \(v\) is bounded away from zero on \(K\), and \(\mathbf u\) is bounded on \(K\), all prefactors are uniformly bounded. Therefore, there exists \(C_{K,\varepsilon}>0\) such that, for every \(x\in K\),
\[
\mathbb{E}\left[
\left\|
\szprox_{\lambda,f}^{\delta,N}(x)
-
\zprox_{\lambda,f}^{\delta}(x)
\right\|^2
\mathbf 1_{\Omega_{x,\varepsilon}}
\right]
\le
\frac{C_{K,\varepsilon}}{N}.
\]
\end{proof}

The proof of Theorem~\ref{thm:S-ZOPPA} is then an application of \cite[Theorem~2.5]{combettes2015stochastic} and \cite[Theorem~4.1]{combettes2026} to the sequence \(\left(x^k\right)_{k\in\mathbb{N}}\).

First note that the sequence generated by Algorithm \ref{alg:szd} can be written as
    \[\forall k\in\mathbb{N},\quad x^{k+1}=x^k+\lambda_k(T_{\lambda,\delta}(x^k)+e^k-x^k),\]
    with \(\lambda_k=1\), \(T_{\lambda,\delta}=\zprox^{\delta}_{\lambda, f}\) which is firmly non expansive from Theorem~\ref{thm:zoppa_cvx} and \(e_k=\szprox^{\delta,N_k}_{\lambda, f}(x^k)-\zprox^{\delta}_{\lambda, f}(x^k)\) for any \(k\in\mathbb{N}\). From the boundedness of \(\left(x^k\right)_{k\in\mathbb{N}}\), we get that there exists a compact set \(K\subset\R^d\) such that for any \(k\in\mathbb{N}\), \(x^k\in K\). Moreover, the normalization assumption and Lemma~\ref{lem:boundedness} guarantee that there exists \(C_{K,\varepsilon}>0\) such that for all \(k\in\mathbb{N}\),
    \[\mathbb{E}\left[\|e_k\|^2|x^k\right]\le\frac{C_{K,\varepsilon}}{N_k}.\]
    By choosing the number of samples such that \(\sum_{k\in\mathbb{N}}N_k^{-1/2}<\infty\), we can therefore ensure that
    \[\sum_{k\in\mathbb{N}}\sqrt{\mathbb{E}\left[\|e_k\|^2|x^k\right]}<\infty.\]
    Finally, since \(T_{\lambda,\delta}\) is firmly non expansive from Theorem~\ref{thm:zoppa_cvx}, we have that for any fixed point \(z\in\R^d\) of \(T_{\lambda,\delta}\),
    \[\left\|T_{\lambda,\delta}(x^k)-z\right\|^2\le \|x^k-z\|^2-\|T_{\lambda,\delta}(x^k)-x^k\|^2.\]
    We get the desired conclusion by observing that $\Fix T_{\lambda,\delta} = \argmin f^{\lambda,\delta}$ and applying \cite[Theorem~2.5]{combettes2015stochastic} and \cite[Theorem~4.1]{combettes2026} (or more indirectly \cite[Proposition~5.1]{combettes2015stochastic}).

The technical assumptions in Theorem~\ref{thm:S-ZOPPA} regarding the sequence
\((x^k)_{k\in\mathbb{N}}\) are used to control the variance of the stochastic proximal estimator.
The following proposition gives a quantitative estimate showing that, along bounded trajectories,
the empirical normalization remains bounded away from zero with high probability.

\begin{proposition}[Stability of the sampled normalization constant]\label{prop:stability}
Let \(f:\mathbb{R}^d\to\mathbb{R}\) satisfy Assumption~\ref{ass:on_f} and assume that
\(f\ge \underline f\). Let
\[
    \bar Z_{\lambda,\delta}(x)
    :=
    \mathbb{E}_{Y\sim\mathcal N(x,\lambda\delta I)}
    \left[e^{-f(Y)/\delta}\right]
    =
    (2\pi\lambda\delta)^{-d/2}Z_{\lambda,\delta}(x).
\]
Assume that the sequence \((x^k)_{k\in\mathbb{N}}\) generated by S-ZOPPA remains in
\(B_M(0)\), and set
\[
    \underline Z_M
    :=
    \inf_{\|x\|\le M}\bar Z_{\lambda,\delta}(x)>0.
\]
Then, for every \(\varepsilon\in(0,\underline Z_M)\), and by letting $C:=e^{-\underline f/\delta}$ we have for every \(k\in\mathbb{N}\)
\[
\mathbb{P}\left(
    \frac1{N_k}\sum_{i=1}^{N_k}e^{-f(y_i^k)/\delta}
    <
    \varepsilon
\right)
\le
\exp\left(
-\frac{2N_k(\underline Z_M-\varepsilon)^2}{C^2}
\right).
\]
\end{proposition}

\begin{proof}
Let
\[
    w_i^k:=e^{-f(y_i^k)/\delta},
    \qquad
    y_i^k\sim\mathcal N(x^k,\lambda\delta I).
\]
Conditionally on \(x^k\), the variables \(w_i^k\) are i.i.d. and satisfy
\[
    0<w_i^k\le C:=e^{-\underline f/\delta}.
\]
Moreover,
\[
    \mathbb{E}[w_i^k\mid x^k]
    =
    \bar Z_{\lambda,\delta}(x^k).
\]
Since \(x^k\in B_M(0)\), we have
\[
    \bar Z_{\lambda,\delta}(x^k)\ge \underline Z_M>\varepsilon.
\]
By Hoeffding's inequality \cite[Theorem 2.8]{Boucheron2013}, conditionally on \(x^k\),
\[
\begin{aligned}
\mathbb{P}\left(
    \frac1{N_k}\sum_{i=1}^{N_k}w_i^k<\varepsilon
    \,\middle|\,x^k
\right)
&\le
\exp\left(
-\frac{2N_k(\bar Z_{\lambda,\delta}(x^k)-\varepsilon)^2}{C^2}
\right)  \\
&\le
\exp\left(
-\frac{2N_k(\underline Z_M-\varepsilon)^2}{C^2}
\right).
\end{aligned}
\]
Thus the conclusion holds with
\[
    c_{M,\varepsilon}
    :=
    \frac{2(\underline Z_M-\varepsilon)^2}{C^2}.
\]
Finally, by the tower property of conditional expectation,
\[
\begin{aligned}
\mathbb{P}\left(
    \frac1{N_k}\sum_{i=1}^{N_k}w_i^k<\varepsilon
\right)
&=
\mathbb{E}\left[
\mathbb{P}\left(
    \frac1{N_k}\sum_{i=1}^{N_k}w_i^k<\varepsilon
    \,\middle|\, x^k
\right)
\right] \\
&\le
\mathbb{E}\left[
\exp(-c_{M,\varepsilon}N_k)
\right] \\
&=
\exp(-c_{M,\varepsilon}N_k).
\end{aligned}
\]
\end{proof}

\section{Accurate approximation of ZOPO for \(x\mapsto |x|\)}
\label{app:zopo_||}

The Zeroth Order Proximal Operator of \(f:x\mapsto|x|\) can be written as
\begin{equation}
    \zprox^{\delta}_{\lambda, f}(x) = \frac{\int_{\mathbb{R}} y \exp\left(-\frac{|y|}{\delta}\right) \exp\left(-\frac{(y-x)^2}{2\lambda\delta}\right) dy}{\int_{\mathbb{R}} \exp\left(-\frac{|y|}{\delta}\right) \exp\left(-\frac{(y-x)^2}{2\lambda\delta}\right) dy}.
\end{equation}

Evaluating these integrals yields the exact closed-form expression:
\begin{equation}
    \label{eq:zopo_unstable}
    \zprox^{\delta}_{\lambda, f}(x) = x + \lambda \left[ \frac{\exp\left(\frac{x}{\delta}\right) \operatorname{erfc}(z_1) - \exp\left(-\frac{x}{\delta}\right) \operatorname{erfc}(z_2)}{\exp\left(\frac{x}{\delta}\right) \operatorname{erfc}(z_1) + \exp\left(-\frac{x}{\delta}\right) \operatorname{erfc}(z_2)} \right],
\end{equation}
where \(\operatorname{erfc}:z\mapsto \frac{2}{\sqrt{\pi}}\int_z^{+\infty}\exp(-t^2)dt\), \(z_1 = \frac{x + \lambda}{\sqrt{2\lambda\delta}}\) and \(z_2 = \frac{\lambda - x}{\sqrt{2\lambda\delta}}\).

To resolve the potential stability issue due to the exponential terms, we can use the scaled complementary error function, \(\operatorname{erfcx}:z\mapsto \exp(z^2)\operatorname{erfc}(z)\). Substituting \(\operatorname{erfc}(z) = \exp(-z^2)\operatorname{erfcx}(z)\) into Equation~\eqref{eq:zopo_unstable} allows the problematic exponential terms to factor out completely. This algebraic cancellation yields a mathematically equivalent but numerically robust formulation:
\begin{equation}
    \label{eq:zopo_stable}
    \zprox^{\delta}_{\lambda, f}(x) = x + \lambda \left[ \frac{\operatorname{erfcx}(z_1) - \operatorname{erfcx}(z_2)}{\operatorname{erfcx}(z_1) + \operatorname{erfcx}(z_2)} \right].
\end{equation}
By exploiting the function \texttt{erfcx} from Scipy, this formulation allows us to compute \(\zprox^{\delta}_{\lambda, f}\) accurately for small values of \(\delta\) in Figure~\ref{fig:intro_sampling}.

\section{Asymptotic and non-asymptotic behavior of S-ZOPO for parabolas}
\label{app:asymptotic_behavior_x2}

We start by writing explicitly the asymptotic rate of the bias and the variance of S-ZOPO, as demonstrated respectively by \cite{hesterberg1988advances} and \cite{agapiou2017importance}
\begin{equation}
    \mathbb{E}\left[\szprox^{\delta,N}_{\lambda, f}(x)\right]=\zprox^{\delta}_{\lambda, f}(x)+\frac{C_1(x,\lambda,\delta)}{N}+o(N^{-1}),
\end{equation}
and 
\begin{equation}
    \mathbb{E}\left\|\szprox^{\delta,N}_{\lambda, f}(x)-\zprox^{\delta}_{\lambda, f}(x)\right\|^2=\frac{C_2(x,\lambda,\delta)}{N}+o(N^{-1}),
\end{equation}
where, denoting \(T=\zprox^{\delta}_{\lambda, f}\), 
{\small\[C_1(x,\lambda,\delta)=\frac{T(x)\mathrm{Var}_{y\in\mathcal{N}(x,\lambda\delta I)}(\exp(-f(y)/\delta))-\mathrm{Cov}_{y\in\mathcal{N}(x,\lambda\delta I)}\left[y\exp(-f(y)/\delta)),\exp(-f(y)/\delta))\right]}{\mathbb{E}_{y\sim\mathcal{N}(x,\lambda\delta I)}\left[\exp(-f(y)/\delta)\right]^2},\]}
and  \[C_2(x,\lambda,\delta)=\frac{\mathbb{E}_{y\sim\mathcal{N}(x,\lambda\delta I)}\left[\exp(-2f(y)/\delta))\|y-T(x)\|^2\right]}{\left(\mathbb{E}_{y\sim\mathcal{N}(x,\lambda\delta I)}\left[\exp(-f(y)/\delta))\right]\right)^2}.\]

\paragraph{The quadratic case.} To study how \(\szprox^{\delta,N}_{\lambda, f}\) behaves, we propose to look into a simple case for which \(\zprox^{\delta}_{\lambda, f}\) and all the related quantities are explicitly computable, i.e.
\[f:y\mapsto \frac{1}{2C}\|y-\mu\|^2,\quad \mu\in\R^d.\]
Indeed, elementary computations show that
\[\left(2\pi\lambda\delta\right)^{-\frac{d}{2}}\int_{\R^d}\exp\left(-\frac{\|y-\mu\|^2}{2C\delta}\right)\exp\left(-\frac{\|y-x\|^2}{2\lambda\delta}\right)dy=\left(\frac{C}{C+\lambda}\right)^\frac{d}{2}\exp\left(-\frac{\|x-\mu\|^2}{2\delta(C+\lambda)}\right),\]
and
\[\left(2\pi\lambda\delta\right)^{-\frac{d}{2}}\int_{\R^d}y\exp\left(-\frac{\|y-\mu\|^2}{2C\delta}\right)\exp\left(-\frac{\|y-x\|^2}{2\lambda\delta}\right)dy=\frac{\lambda\mu+Cx}{C+\lambda}\left(\frac{C}{C+\lambda}\right)^\frac{d}{2}\exp\left(-\frac{\|x-\mu\|^2}{2\delta(C+\lambda)}\right).\]
It follows that for any \(x\in\R^d\), \(\zprox^{\delta}_{\lambda, f}(x)=\frac{\lambda\mu+Cx}{C+\lambda}=\prox_{\lambda f}(x)\). It is also possible to compute exactly the constants \(C_1(x,\lambda,\delta)\) and \(C_2(x,\lambda,\delta)\) appearing in the asymptotic expressions of the bias and the variance of \(\szprox^{\delta,N}_{\lambda, f}(x)\). Following tedious but elementary computations, we get that
\[C_1(x,\lambda,\delta)=\frac{\lambda(x-\mu)}{(2\lambda+C)(\lambda+C)}\phi(x,\mu,\lambda,\delta,C),\]
and
\[C_2(x,\lambda,\delta)=\left(\frac{\lambda^2C^2}{(2\lambda+C)^2(\lambda+C)^2}\|x-\mu\|^2+\frac{d\lambda\delta C}{2\lambda+C}\right)\phi(x,\mu,\lambda,\delta,C),\]
where \[\phi(x,\mu,\lambda,\delta,C)=\left(\frac{\lambda+C}{\sqrt{C(2\lambda+C)}}\right)^d\exp\left(\frac{\lambda\|x-\mu\|^2}{\delta(2\lambda+C)(\lambda+C)}\right).\]

It appears that these constants depend on a quantity \(\phi\) which grows exponentially with \(\frac{\|x-\mu\|^2}{\delta}\), unveiling a significant variance and bias when \(x\) is far from the minimizer \(\mu\) or \(\delta\) is small. However, as shown in Figure~\ref{fig:the_vs_emp}, the behavior of the variance predicted by this asymptotic estimate does not hold for fixed \(N\) when \(\|x-\mu\|\) grows.

\begin{figure}[H]
    \centering
    \includegraphics[width=0.9\linewidth]{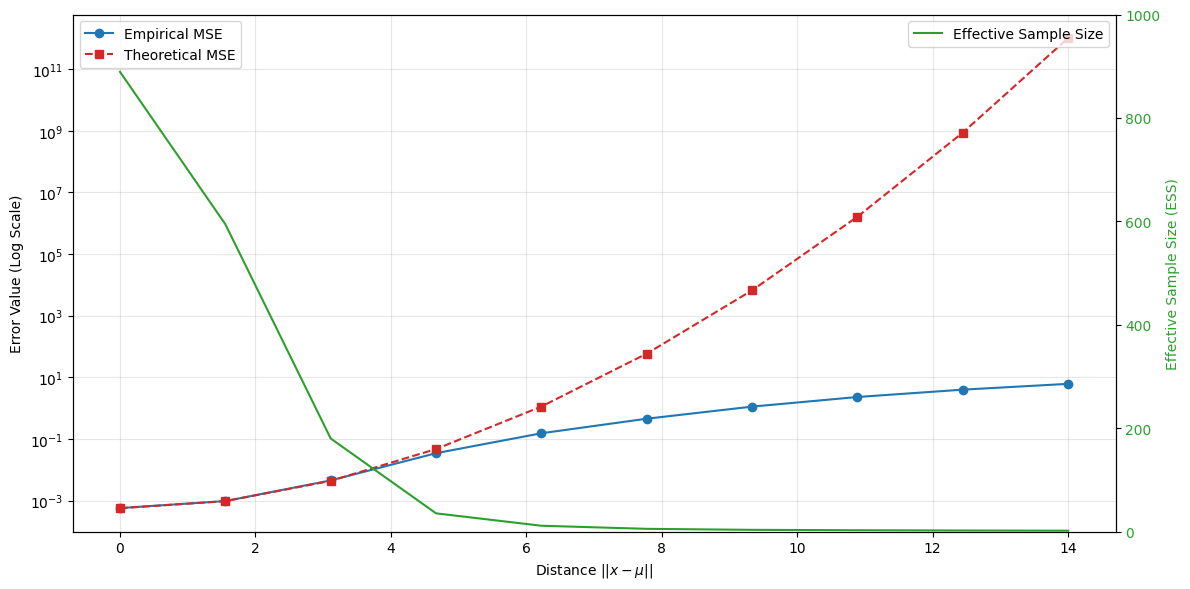}
    \caption{Theoretical asymptotic variance and empirical variance of \(\szprox^{\delta,N}_{\lambda, f}(x)\) w.r.t. to the distance to the minimizer for a quadratic \(f\). The empiric ESS is displayed in green with the corresponding values on the right axis. The constants used here are \(d=2\), \(\delta = 1\), \(\lambda = 0.5\) and \(C = 1\), for \(N=1000\) samples and \(500\) trials per point.}
    \label{fig:the_vs_emp}
\end{figure}

\section{Experiments settings}\label{app:exp_setting}

All simulations were implemented in Python and run on a laptop with an Intel Core i7-1165G7 CPU @ 2.80GHz and 8 GB of RAM.

\subsection{Experiment of Section \ref{sec:exp_lambdarole}}\label{app:exp_setting_lambda} For the generation of Figure~\ref{fig:varying_lambda} we have chosen the following setting. We set $\delta = 1$, $N = 10000$ and total iteration number of $K_{\max}=1000$. As 1D function we have set $f(x) = 0.3 x^2 + \sin(4x) + 0.5 \cos(20x)$ and the initial point was set to $x^0=3$. For the benchmark problems we have set $d=10$ and performed $20$ different runs with different initializations and reported the error bars in Figure~\ref{fig:varying_lambda}. The values for $\lambda$ fixed are $0.01, 0.1, 1$ and $10$ while the schedule for the decreasing $\lambda$ are in a geometric scale from $10$ to $0.01$ with $10$ intermidiate values.

Figure~\ref{fig:varying_lambda} consistently shows that letting $\lambda$ vary from $10$ to $0.01$ is beneficial for the algorithm. This is numerical evidence showing that decreasing $\lambda$ eventually leads the algorithm to stationary points of $f$ and helps escaping from poor stationary points. This is further backed up by our theory, since we prove that for small $\lambda$ the final points are almost stationary for $f$ in Theorem~\ref{thm:rates} and we prove that for large $\lambda$ the landscape of $f^{\lambda,\delta}$ attains a convexification effect in Theorem~\ref{thm:convexification}.

\subsection{Experiment setting for Section \ref{sec:sampling_challenges} and the introduction}\label{app:exp_setting_delta}

For the experiments in Figure~\ref{fig:intro_sampling} and Figure~\ref{fig:bias+var_vs_N}, we consider the one-dimensional function
\[
    f:x\mapsto |x|,
\]
with initial point \(x=2\) and proximal parameter \(\lambda=1\). The exact proximal point is therefore
\[
    \prox_f(2)=1.
\]
The reference value of \(\zprox^\delta_{1,f}(2)\) is computed using the closed-form expression derived in Appendix~\ref{app:zopo_||}, relying on the numerically stable scaled complementary error function \(\operatorname{erfcx}\), implemented through \texttt{scipy.special.erfcx}.

For the S-ZOPO approximation, samples are drawn according to
\[
    y_i\sim \mathcal N(2,\delta),
\]
and the estimator is computed as
\[
    \szprox^{\delta,N}_{1,f}(2)
    =
    \frac{\sum_{i=1}^N y_i\exp(-|y_i|/\delta)}
    {\sum_{i=1}^N \exp(-|y_i|/\delta)}.
\]
The effective sample size is computed as
\[
    \widehat{\operatorname{ESS}}
    =
    \frac{\left(\sum_{i=1}^N w_i\right)^2}
    {\sum_{i=1}^N w_i^2},
    \qquad
    w_i=\exp(-|y_i|/\delta).
\]

In Figure~\ref{fig:intro_sampling}, we use \(N=5000\) samples and report the average over \(100\) independent trials. In Figure~\ref{fig:bias+var_vs_N}, the empirical bias and variance are computed over \(100\) independent trials for \(N\in\{50,500,5000\}\). The values of \(\delta\) are sampled logarithmically.

\end{document}